\documentclass[10pt,a4paper]{article}

\usepackage{amsfonts}
\usepackage{amssymb}
\usepackage{amsthm}
\usepackage[english]{babel}
\usepackage{hhline}
\usepackage[utf8]{inputenc}
\usepackage{amsmath}
\usepackage[all,2cell,ps]{xy}
\usepackage{qsymbols}
\usepackage{color}
\usepackage{epsfig}
\usepackage{color}
\usepackage{graphics}
\usepackage{graphicx}%
\usepackage{enumerate}
\usepackage{tikz}
\usepackage[bottom]{footmisc}
\usepackage{authblk}
\usepackage{multicol}

\numberwithin{equation}{section}

\usepackage[colorlinks,linkcolor=black,citecolor=blue,urlcolor=blue]{hyperref}
\theoremstyle{plain}

\newtheorem{theorem}{Theorem}
\newtheorem{lemma}[theorem]{Lemma}

\newtheorem{definition}[theorem]{Definition}

\newcommand{\RNA}{\mathsf{RNA}}
\newcommand{\ra}{\rightarrow}

\newcommand{\we}{\wedge}

\newcommand{\s}{\sim}

\title{An independent equational basis for the variety of reflexive Nelson algebras}

\author{Cornejo J.M.,  Helt P. S. and San Mart\'in H.J.}

\date{}

\begin{document}
	
\maketitle

%\section{NUEVOOOO}

\begin{abstract}
In this manuscript, we provide an independent equational basis for the variety of reflexive Nelson algebras, 
a generalization of the variety of SNA-algebras.
The proof of this result relies on a substantial number of technical arguments and computational examples.
The result complements a broader study of reflexive Nelson algebras by showing that the variety admits an 
independent equational axiomatization.
\end{abstract}

\section{Introduction}

In what follows, we present the definitions of reflexive Nelson algebras and R-algebras. 
The aim of this manuscript is to prove the following two results:
\begin{itemize}
\item The varieties of reflexive Nelson algebras and R-algebras coincide.
\item The set of identities given in Definition \ref{def_R}, which defines the variety of R-algebras, is independent. 
Therefore, it constitutes an independent equational basis for the variety.
\end{itemize}

\begin{definition} \label{def_R}
	An algebra $\langle T, \we, \vee, \ra, \s, 0, 1\rangle$ of type $(2, 2, 2, 1, 0, 0)$ is said to be an 
	\textit{$R$-algebra} if it satisfies the following identities:
	\begin{enumerate}
		\item[$\mathrm{(R1)}$] $x \wedge (y \vee z) \approx (z \wedge x) \vee (y \wedge x)$ \label{lattice2}
		\item[$\mathrm{(R2)}$] $x \approx x \we 1$ \label{top}
		\item[$\mathrm{(R3)}$] $0 \approx 0\we x$ \label{bottom}
		\item[$\mathrm{(R4)}$] $\sim \sim x \approx= x$, \label{dobleneg}
		\item[$\mathrm{(R5)}$] $\sim (x\we y) \approx \sim x \vee \sim y$, \label{negsup}
		\item[$\mathrm{(R6)}$] $(x\vee y)\ra z \approx (x\ra z)\we (y\ra z)$,\label{supremo}
		\item[$\mathrm{(R7)}$] $z\ra (x\we y) \approx (z\ra x)\we (z\ra y)$,\label{infimo}
		\item[$\mathrm{(R8)}$] $((x\ra y)\we (y\ra z))\ra (x\ra z) \approx 1$, \label{trans}
		\item[$\mathrm{(R9)}$] $x\we (x\ra y)\leq x\we (\s x\vee y)$, \label{inf}
		\item[$\mathrm{(R10)}$] $\s(x\ra y)\ra (x\we \s y)  \approx 1$. \label{au}
	\end{enumerate}
	We denote by $\mathcal R$ the variety of $R$-algebras.
\end{definition}

Recall that a \textit{Kleene algebra} \cite{Vig} is a bounded distributive lattice endowed with a unary operation $\sim$
which satisfies the following identities: 
\begin{enumerate}[\normalfont 1)]
\item $\sim \sim x \approx  x$,
\item $\sim (x\we y) \approx  \sim x \vee \sim y$,
\item $(x\we \sim x)\we (y \vee \sim y) \approx  x\we \sim x$.
\end{enumerate}

\begin{definition}\label{generalised}
	An algebra $\langle T, \we, \vee, \ra, \s, 0, 1\rangle$ of type $(2, 2, 2, 1, 0, 0)$ is said to be a
	\textit{reflexive Nelson algebra} (R-Nelson algebra for short) 
	if $\langle T, \wedge, \vee, \sim, 0, 1\rangle$ is a Kleene algebra 
	which satisfies the identities $\mathrm{(R6)}$,  $\mathrm{(R7)}$,  $\mathrm{(R8)}$, 
	$\mathrm{(R9)}$,  $\mathrm{(R10)}$ 
	of Definition \ref{def_R}, and the following two additional identities:
	\begin{enumerate}[\normalfont a)]
		\item $x\ra x \approx 1$, %\label{imp1}
		\item $(x\we \s y)\ra \s(x\ra y) \approx 1$. %\label{ultima}
	\end{enumerate}
	We denote by $\RNA$ the variety of R-Nelson algebras.
\end{definition}

An R-Nelson algebra is said to be an \textit{SNA-algebra} if this satisfies the additional identity 
$x\ra y\leq z\ra (x\ra y)$. The variety of SNA-álgebras were introduced and studied in \cite{LMSM} with the 
aim to generalize the well known relation 
between Heyting algebras and Nelson algebras in the framework of subresiduated lattices \cite{EH}.

\section{The varieties $\RNA$ and $\mathcal R$ coincide}

The following two lemmas involve computational complexity considerations. 

\begin{lemma} \label{lemaDistrib}
Let $\mathbf T = \langle T, \we, \vee, \ra, \s, 0, 1\rangle$ be an R-algebra. 
Then $\langle T, \we, \vee, 0, 1\rangle$ is a bounded distributive lattice.	
\end{lemma}

\begin{proof}
	Let $a,b,c \in T$. Since
	%$$
	%\begin{array}{lcll}
	%\sim a	& = & \sim (a \wedge 1) & \mbox{by (R\ref{top})} \\
	%	& = & \sim a \vee \sim 1 & \mbox{by (R\ref{negsup})} 
	%\end{array}
	%$$
	\noindent $\sim a	$
	$\overset{  (R\ref{top}) 
	}{=}  \sim (a \wedge 1) $
	$\overset{  (R\ref{negsup}) 
	}{=}  \sim a \vee \sim 1 $ 
	then
	\begin{equation} \label{hipIdentidad10}
		\mathbf T \models \sim x \approx \sim x \vee \sim 1.
	\end{equation}
	Besides, since
	%$$
	%\begin{array}{lcll}
	%a	& = & \sim \sim a & \mbox{by (R\ref{dobleneg})} \\
	%	& = & \sim \sim a \vee \sim 1 & \mbox{by (\ref{hipIdentidad10})} \\
	%	& = & a \vee \sim 1 & \mbox{by (R\ref{dobleneg})} 
	%\end{array}
	%$$
	\noindent $a	$
	$\overset{  (R\ref{dobleneg}) 
	}{=}  \sim \sim a $
	$\overset{  (\ref{hipIdentidad10}) 
	}{=}  \sim \sim a \vee \sim 1 $
	$\overset{  (R\ref{dobleneg}) 
	}{=}  a \vee \sim 1 $ 
	then
	\begin{equation} \label{identidad10}
		\mathbf T \models x \approx x \vee \sim 1.
	\end{equation}
	The condition 
	\begin{equation} \label{identidad12}
		\mathbf T \models 1 \wedge (x \wedge (y \vee z)) \approx x \wedge (1 \wedge (y \vee z))
	\end{equation}
	holds because
	%$$
	%\begin{array}{lcll}
	%1 \wedge (a \wedge (b \vee c))	& = & 1 \wedge [(c \wedge a) \vee (b \wedge a)] & \mbox{by (R\ref{lattice2})} \\
	%	& = & [(b \wedge a) \wedge 1] \vee [(c \wedge a) \wedge 1]  & \mbox{by (R\ref{lattice2})} \\
	%	& = & (b \wedge a) \wedge (c \wedge a) & \mbox{by (R\ref{top})} \\
	%	& = & a \wedge (c \vee b) & \mbox{by (R\ref{lattice2})} \\
	%	& = & a \wedge ((c \wedge 1) \vee (b \wedge 1)) & \mbox{by (R\ref{top})} \\
	%	& = & a \wedge [1 \wedge (b \vee c)] & \mbox{by (R\ref{lattice2})}
	%\end{array}
	%$$
	\noindent $1 \wedge (a \wedge (b \vee c))	$
	$\overset{  (R\ref{lattice2}) 
	}{=}  1 \wedge [(c \wedge a) \vee (b \wedge a)] $
	$\overset{  (R\ref{lattice2}) 
	}{=}  [(b \wedge a) \wedge 1] \vee [(c \wedge a) \wedge 1]  $
	$\overset{  (R\ref{top}) 
	}{=}  (b \wedge a) \wedge (c \wedge a) $
	$\overset{  (R\ref{lattice2}) 
	}{=}  a \wedge (c \vee b) $
	$\overset{  (R\ref{top}) 
	}{=}  a \wedge ((c \wedge 1) \vee (b \wedge 1)) $
	$\overset{  (R\ref{lattice2})
	}{=}  a \wedge [1 \wedge (b \vee c)] $. 
	Also note that
	%$$
	%\begin{array}{lcll}
	%a	& = & a \wedge 1 & \mbox{by (R\ref{top})} \\
	%	& = & a \wedge (1 \wedge 1) & \mbox{by (R\ref{top})} \\
	%	& = & a \wedge (1 \wedge (1 \vee (\sim 1))) & \mbox{by (\ref{identidad10})} \\
	%	& = & 1 \wedge (a \wedge (1 \vee \sim 1)) & \mbox{by (\ref{identidad12})} \\
	%	& = & 1 \wedge (a \wedge 1) & \mbox{by (\ref{identidad10})} \\
	%	& = & 1 \wedge a & \mbox{by (R\ref{top})} 
	%\end{array}
	%$$
	\noindent $a	$
	$\overset{  (R\ref{top}) 
	}{=}  a \wedge 1 $
	$\overset{  (R\ref{top}) 
	}{=}  a \wedge (1 \wedge 1) $
	$\overset{  (\ref{identidad10}) 
	}{=}  a \wedge (1 \wedge (1 \vee (\sim 1))) $
	$\overset{  (\ref{identidad12}) 
	}{=}  1 \wedge (a \wedge (1 \vee \sim 1)) $
	$\overset{  (\ref{identidad10}) 
	}{=}  1 \wedge (a \wedge 1) $
	$\overset{  (R\ref{top}) 
	}{=}  1 \wedge a, $
	so
	\begin{equation} \label{identidad8}
		\mathbf T \models 1 \wedge x \approx x. 
	\end{equation}
	On other hand,
	%$$
	%\begin{array}{lcll}
	%a \vee b	& = & 1 \wedge (a \vee b) & \mbox{by (\ref{identidad8})} \\
	%	& = & (b \wedge 1) \vee (a \wedge 1) & \mbox{by (R\ref{lattice2})} \\
	%	& = & b \vee a & \mbox{by (R\ref{top})} 
	%\end{array}
	%$$
	\noindent $a \vee b	$
	$\overset{  (\ref{identidad8}) 
	}{=}  1 \wedge (a \vee b) $
	$\overset{  (R\ref{lattice2}) 
	}{=}  (b \wedge 1) \vee (a \wedge 1) $
	$\overset{  (R\ref{top}) 
	}{=}  b \vee a $.
	Then
	\begin{equation} \label{identidad5}
		\mathbf T \models x \vee y \approx y \vee x. 
	\end{equation}
	The identity (\ref{identidad5}) allows to show that
	\begin{equation} \label{identidad3}
		\mathbf T \models x \wedge y \approx y \wedge x
	\end{equation}
	such that
	%$$
	%\begin{array}{lcll}
	%a \wedge b	& = & \sim \sim (a \wedge b) & \mbox{by (R\ref{dobleneg})} \\
	%	& = & \sim (\sim a \vee \sim b) & \mbox{by (R\ref{negsup})} \\
	%	& = & \sim (\sim b \vee \sim a) & \mbox{by (\ref{identidad5})} \\
	%	& = & \sim \sim (b \wedge a)  & \mbox{by (R\ref{negsup})} \\
	%	& = & b \wedge a & \mbox{by (R\ref{dobleneg})} 
	%\end{array}
	%$$
	\noindent $a \wedge b	$
	$\overset{  (R\ref{dobleneg}) 
	}{=}  \sim \sim (a \wedge b) $
	$\overset{  (R\ref{negsup}) 
	}{=}  \sim (\sim a \vee \sim b) $
	$\overset{  (\ref{identidad5}) 
	}{=}  \sim (\sim b \vee \sim a) $
	$\overset{  (R\ref{negsup}) 
	}{=}  \sim \sim (b \wedge a)  $
	$\overset{  (R\ref{dobleneg}) 
	}{=}  b \wedge a $. 
	
	Note that
	%$$
	%\begin{array}{lcll}
	%1 \to 0	& = & 1 \wedge (1 \to 0) & \mbox{by (\ref{identidad8})} \\
	%	& = & [1 \wedge (1 \to 0)] \wedge [1 \wedge (\sim 1 \vee 0)] & \mbox{by (R\ref{inf})} \\
	%	& = & (1 \to 0) \wedge (\sim 1 \vee 0) & \mbox{by (\ref{identidad8})} \\
	%	& = & (1 \to 0) \wedge (0 \vee \sim 1) & \mbox{by (\ref{identidad5})} \\
	%	& = & (1 \to 0) \wedge 0 & \mbox{by (\ref{identidad10})} \\
	%	& = & 0 \wedge (1 \to 0) & \mbox{by (\ref{identidad3})} \\
	%	& = & 0 & \mbox{by (R\ref{bottom})}
	%\end{array}
	%$$
	\noindent $1 \to 0	$
	$\overset{  (\ref{identidad8}) 
	}{=}  1 \wedge (1 \to 0) $
	$\overset{  (R\ref{inf}) 
	}{=}  [1 \wedge (1 \to 0)] \wedge [1 \wedge (\sim 1 \vee 0)] $
	$\overset{  (\ref{identidad8}) 
	}{=}  (1 \to 0) \wedge (\sim 1 \vee 0) $
	$\overset{  (\ref{identidad5}) 
	}{=}  (1 \to 0) \wedge (0 \vee \sim 1) $
	$\overset{  (\ref{identidad10}) 
	}{=}  (1 \to 0) \wedge 0 $
	$\overset{  (\ref{identidad3}) 
	}{=}  0 \wedge (1 \to 0) $
	$\overset{  (R\ref{bottom})
	}{=}  0 $. 
	Hence,
	\begin{equation} \label{identidad13}
		1 \to 0 = 0.
	\end{equation}
	
	Taking into account that
	%$$
	%\begin{array}{lcll}
	%0 \to 0	& = & 0 \to (1 \to 0) & \mbox{by (\ref{identidad13})} \\
	%	& = & [0 \wedge (0 \to 0)] \to (1 \to 0)  & \mbox{by (R\ref{bottom})} \\
	%	& = & [(1 \to 0) \wedge (0 \to 0)] \to (1 \to 0)  & \mbox{by (\ref{identidad13})} \\
	%	& = & 1 & \mbox{by (R\ref{inf})} 
	%\end{array}
	%$$
	\noindent $0 \to 0	$
	$\overset{  (\ref{identidad13}) 
	}{=}  0 \to (1 \to 0) $
	$\overset{  (R\ref{bottom}) 
	}{=}  [0 \wedge (0 \to 0)] \to (1 \to 0)  $
	$\overset{  (\ref{identidad13}) 
	}{=}  [(1 \to 0) \wedge (0 \to 0)] \to (1 \to 0)  $
	$\overset{  (R\ref{inf}) 
	}{=}  1 $ 
	we get
	\begin{equation} \label{identidad9}
		0 \to 0 = 1.
	\end{equation}
	Also, taking into account that
	%$$
	%\begin{array}{lcll}
	%1 \wedge [(1 \wedge a) \wedge (1 \vee b)]	& = & 1 \wedge [[b \wedge (1 \wedge a)] \vee [1 \wedge (1 \wedge a)]] & \mbox{by (R\ref{lattice2})} \\
	%	& = & 1 \wedge [(b \wedge a) \vee (1 \wedge a)]  & \mbox{by (\ref{identidad8})} \\
	%	& = & [(1 \wedge a) \wedge 1] \vee [(b \wedge a) \wedge 1] & \mbox{by (R\ref{lattice2})} \\
	%	& = & [a \wedge 1] \vee [(b \wedge a) \wedge 1] & \mbox{by (\ref{identidad8})} \\
	%	& = & [a \wedge 1] \vee (b \wedge a) & \mbox{by (R\ref{top})} \\
	%	& = & [a \wedge 1] \vee (b \wedge (1 \wedge a)) & \mbox{by (\ref{identidad8})} \\
	%	& = & a \vee (b \wedge (1 \wedge a)) & \mbox{by (R\ref{top})} 
	%\end{array}
	%$$ 
	\noindent $1 \wedge [(1 \wedge a) \wedge (1 \vee b)]	$
	$\overset{  (R\ref{lattice2}) 
	}{=}  1 \wedge [[b \wedge (1 \wedge a)] \vee [1 \wedge (1 \wedge a)]] $
	$\overset{  (\ref{identidad8}) 
	}{=}  1 \wedge [(b \wedge a) \vee (1 \wedge a)]  $
	$\overset{  (R\ref{lattice2}) 
	}{=}  [(1 \wedge a) \wedge 1] \vee [(b \wedge a) \wedge 1] $
	$\overset{  (\ref{identidad8}) 
	}{=}  [a \wedge 1] \vee [(b \wedge a) \wedge 1] $
	$\overset{  (R\ref{top}) 
	}{=}  [a \wedge 1] \vee (b \wedge a) $
	$\overset{  (\ref{identidad8}) 
	}{=}  [a \wedge 1] \vee (b \wedge (1 \wedge a)) $
	$\overset{  (R\ref{top}) 
	}{=}  a \vee (b \wedge (1 \wedge a)) $ 
	we get
	\begin{equation} \label{identidad11}
		\mathbf T \models 1 \wedge [(1 \wedge x) \wedge (1 \vee y)] \approx x \vee (y \wedge (1 \wedge x))
	\end{equation}
	Since
	%$$
	%\begin{array}{lcll}
	%a \vee (b \wedge a)	& = & a \vee (b \wedge (1 \wedge a)) & \mbox{by (\ref{identidad8})} \\
	%	& = & 1 \wedge [(1 \wedge a) \wedge (1 \vee b)] & \mbox{by (\ref{identidad11})} \\
	%	& = & a \wedge (1 \vee b) & \mbox{by (\ref{identidad8})}
	%\end{array}
	%$$
	\noindent $a \vee (b \wedge a)	$
	$\overset{  (\ref{identidad8}) 
	}{=}  a \vee (b \wedge (1 \wedge a)) $
	$\overset{  (\ref{identidad11}) 
	}{=}  1 \wedge [(1 \wedge a) \wedge (1 \vee b)] $
	$\overset{  (\ref{identidad8})
	}{=}  a \wedge (1 \vee b) $ 
	then we have that
	\begin{equation} \label{identidad4}
		\mathbf T \models x \vee (y \wedge x) \approx x \wedge (1 \vee y)
	\end{equation}
	Note that
	%$$
	%\begin{array}{lcll}
	%(\sim 1) \wedge ((\sim 1) \to 0)	& = & [(\sim 1) \wedge ((\sim 1) \to 0)] \wedge [(\sim 1) \wedge ((\sim \sim 1) \vee 0)] & \mbox{by (R\ref{inf})} \\
	%	& = & [(\sim 1) \wedge ((\sim 1) \to 0)] \wedge [(\sim 1) \wedge (1 \vee 0)] & \mbox{by (R\ref{dobleneg})} \\
	%	& = & [(\sim 1) \wedge ((\sim 1) \to 0)] \wedge [(0 \wedge (\sim 1)) \vee (1 \wedge (\sim 1))] & \mbox{by (R\ref{lattice2})} \\
	%	& = & [(\sim 1) \wedge ((\sim 1) \to 0)] \wedge [0 \vee (1 \wedge (\sim 1))] & \mbox{by (R\ref{top})} \\
	%	& = & [(\sim 1) \wedge ((\sim 1) \to 0)] \wedge [0 \vee (\sim 1)] & \mbox{by (\ref{identidad8})} \\
	%	& = & [(\sim 1) \wedge ((\sim 1) \to 0)] \wedge 0 & \mbox{by (\ref{identidad10})} \\
	%	& = & 0 \wedge [(\sim 1) \wedge ((\sim 1) \to 0)] & \mbox{by (\ref{identidad3})} \\
	%	& = & 0 & \mbox{by (R\ref{bottom})} 
	%\end{array}
	%$$
	\noindent $(\sim 1) \wedge ((\sim 1) \to 0)	$
	$\overset{  (R\ref{inf}) 
	}{=}  [(\sim 1) \wedge ((\sim 1) \to 0)] \wedge [(\sim 1) \wedge ((\sim \sim 1) \vee 0)] $
	$\overset{  (R\ref{dobleneg}) 
	}{=}  [(\sim 1) \wedge ((\sim 1) \to 0)] \wedge [(\sim 1) \wedge (1 \vee 0)] $
	$\overset{  (R\ref{lattice2}) 
	}{=}  [(\sim 1) \wedge ((\sim 1) \to 0)] \wedge [(0 \wedge (\sim 1)) \vee (1 \wedge (\sim 1))] $
	$\overset{  (R\ref{top}) 
	}{=}  [(\sim 1) \wedge ((\sim 1) \to 0)] \wedge [0 \vee (1 \wedge (\sim 1))] $
	$\overset{  (\ref{identidad8}) 
	}{=}  [(\sim 1) \wedge ((\sim 1) \to 0)] \wedge [0 \vee (\sim 1)] $
	$\overset{  (\ref{identidad10}) 
	}{=}  [(\sim 1) \wedge ((\sim 1) \to 0)] \wedge 0 $
	$\overset{  (\ref{identidad3}) 
	}{=}  0 \wedge [(\sim 1) \wedge ((\sim 1) \to 0)] $
	$\overset{  (R\ref{bottom}) 
	}{=}  0, $
	so
	\begin{equation} \label{identidad6}
		(\sim 1) \wedge ((\sim 1) \to 0) = 0
	\end{equation}
	Also, since
	%$$
	%\begin{array}{lcll}
	%0	& = & (\sim 1) \wedge ((\sim 1) \to 0) & \mbox{by (\ref{identidad6})} \\
	%	& = & (\sim 1) \wedge [1 \wedge ((\sim 1) \to 0)] & \mbox{by (\ref{identidad8})} \\
	%	& = & (\sim 1) \wedge [(0 \to 0) \wedge ((\sim 1) \to 0)] & \mbox{by (\ref{identidad9})} \\
	%	& = & (\sim 1) \wedge [(0 \vee (\sim 1)) \to 0] & \mbox{by (R\ref{supremo})} \\
	%	& = & (\sim 1) \wedge [0 \to 0] & \mbox{by (\ref{identidad10})} \\
	%	& = & (\sim 1) \wedge 1 & \mbox{by (\ref{identidad9})} \\
	%	& = & \sim 1 & \mbox{by (R\ref{top})}
	%\end{array}
	%$$
	\noindent $0	$
	$\overset{  (\ref{identidad6}) 
	}{=}  (\sim 1) \wedge ((\sim 1) \to 0) $
	$\overset{  (\ref{identidad8}) 
	}{=}  (\sim 1) \wedge [1 \wedge ((\sim 1) \to 0)] $
	$\overset{  (\ref{identidad9}) 
	}{=}  (\sim 1) \wedge [(0 \to 0) \wedge ((\sim 1) \to 0)] $
	$\overset{  (R\ref{supremo}) 
	}{=}  (\sim 1) \wedge [(0 \vee (\sim 1)) \to 0] $
	$\overset{  (\ref{identidad10}) 
	}{=}  (\sim 1) \wedge [0 \to 0] $
	$\overset{  (\ref{identidad9}) 
	}{=}  (\sim 1) \wedge 1 $
	$\overset{  (R\ref{top})
	}{=}  \sim 1 $ 
	then
	\begin{equation} \label{identidad2}
		\sim 1  = 0.
	\end{equation}
	In consequence,
	%$$
	%\begin{array}{lcll}
	%a \vee 1	& = & \sim \sim  a \vee \sim \sim 1 & \mbox{by (R\ref{dobleneg})} \\
	%	& = & \sim (\sim a \wedge \sim 1) & \mbox{by (R\ref{negsup})} \\
	%	& = & \sim (\sim a \wedge 0) & \mbox{by (\ref{identidad2})} \\
	%	& = & \sim (0 \wedge \sim a) & \mbox{by (\ref{identidad3})} \\
	%	& = & \sim 0 & \mbox{by (R\ref{bottom})} \\
	%	& = & \sim \sim 1 & \mbox{by (\ref{identidad2})} \\
	%	& = & 1 & \mbox{by (R\ref{dobleneg})}
	%\end{array}
	%$$
	\noindent $a \vee 1	$
	$\overset{  (R\ref{dobleneg}) 
	}{=}  \sim \sim  a \vee \sim \sim 1 $
	$\overset{  (R\ref{negsup}) 
	}{=}  \sim (\sim a \wedge \sim 1) $
	$\overset{  (\ref{identidad2}) 
	}{=}  \sim (\sim a \wedge 0) $
	$\overset{  (\ref{identidad3}) 
	}{=}  \sim (0 \wedge \sim a) $
	$\overset{  (R\ref{bottom}) 
	}{=}  \sim 0 $
	$\overset{  (\ref{identidad2}) 
	}{=}  \sim \sim 1 $
	$\overset{  (R\ref{dobleneg})
	}{=}  1, $ 
	so
	\begin{equation} \label{hipIdentidad1}
		\mathbf T \models x \vee 1 \approx 1.
	\end{equation}
	Then
	%$$
	%\begin{array}{lcll}
	%a \vee (b \wedge a)	& = & a \wedge (1 \vee b) & \mbox{by (\ref{identidad4})} \\
	%	& = & a \wedge (b \vee 1) & \mbox{by (\ref{identidad5})} \\
	%	& = & a \wedge 1 & \mbox{by (\ref{hipIdentidad1})} \\
	%	& = & a & \mbox{by (R\ref{top})}
	%\end{array}
	%$$
	\noindent $a \vee (b \wedge a)	$
	$\overset{  (\ref{identidad4}) 
	}{=}  a \wedge (1 \vee b) $
	$\overset{  (\ref{identidad5}) 
	}{=}  a \wedge (b \vee 1) $
	$\overset{  (\ref{hipIdentidad1}) 
	}{=}  a \wedge 1 $
	$\overset{  (R\ref{top})
	}{=}  a $. 
	Thus,
	\begin{equation} \label{hip2Identidad1}
		\mathbf T \models  x \vee (y \wedge x) \approx x.
	\end{equation}
	Also, we have that
	%$$
	%\begin{array}{lcll}
	%a \wedge (a \vee b)	& = & \sim \sim [a \wedge (a \vee b)] & \mbox{by (R\ref{dobleneg})} \\
	%	& = & \sim [\sim a \vee \sim (a \vee b)] & \mbox{by (R\ref{negsup})} \\
	%	& = & \sim [\sim a \vee \sim (b \vee a)] & \mbox{by (\ref{identidad5})} \\
	%	& = & \sim [\sim a \vee \sim ((\sim \sim b) \vee (\sim \sim a))] & \mbox{by (R\ref{dobleneg})} \\
	%	& = & \sim [\sim a \vee \sim (\sim (\sim b \wedge \sim a))] & \mbox{by (R\ref{negsup})} \\
	%	& = & \sim [\sim a \vee (\sim b \wedge \sim a)] & \mbox{by (R\ref{dobleneg})} \\
	%	& = & \sim \sim a & \mbox{by (\ref{hip2Identidad1})} \\
	%	& = & a & \mbox{by (R\ref{dobleneg})} 
	%\end{array}
	%$$
	\noindent $a \wedge (a \vee b)	$
	$\overset{  (R\ref{dobleneg}) 
	}{=}  \sim \sim [a \wedge (a \vee b)] $
	$\overset{  (R\ref{negsup}) 
	}{=}  \sim [\sim a \vee \sim (a \vee b)] $
	$\overset{  (\ref{identidad5}) 
	}{=}  \sim [\sim a \vee \sim (b \vee a)] $
	$\overset{  (R\ref{dobleneg}) 
	}{=}  \sim [\sim a \vee \sim ((\sim \sim b) \vee (\sim \sim a))] $
	$\overset{  (R\ref{negsup}) 
	}{=}  \sim [\sim a \vee \sim (\sim (\sim b \wedge \sim a))] $
	$\overset{  (R\ref{dobleneg}) 
	}{=}  \sim [\sim a \vee (\sim b \wedge \sim a)] $
	$\overset{  (\ref{hip2Identidad1}) 
	}{=}  \sim \sim a $
	$\overset{  (R\ref{dobleneg}) 
	}{=}  a. $ 
	Therefore,
	\begin{equation} \label{identidad1}
		\mathbf T \models x \wedge (x \vee y) \approx x.
	\end{equation}
	
The identities  (R\ref{lattice2}) and (\ref{identidad1}) are precisely the axioms given by Sholander in \cite{sholander51postulates} for distributive lattices. In view of (R\ref{top}) and (R\ref{bottom}), $\mathbf T$ is also bounded.

	%@article {sholander51postulates,
		%	AUTHOR = {Sholander, Marlow},
		%	TITLE = {Postulates for distributive lattices},
		%	JOURNAL = {Canadian J. Math.},
		%	FJOURNAL = {Canadian Journal of Mathematics. Journal Canadien de Math\'ematiques},
		%	VOLUME = {3},
		%	YEAR = {1951},
		%	PAGES = {28--30},
		%	ISSN = {0008-414X},
		%	MRCLASS = {09.1X},
		%	MRNUMBER = {0038942 (12,472k)},
		%	MRREVIEWER = {P. M. Whitman},
		%}
	
\end{proof}

In what follows we will use that every R-algebra is a bounded distributive lattice.
In the framework of R-algebras we define $x^* = x\ra 0$ and the binary relation $\leq$ by
by $x \leq y$ if and only if $x = x \we y$ (or, in an equivalent way,
$x \leq y$ if and only if $y = x \vee y$).

\begin{lemma} \label{otherconditions}
Let $\mathbf T = \langle T, \we, \vee, \ra, \s, 0, 1\rangle$ be an R-algebra.
Then the following quasi-identities are satisfied:
	\begin{enumerate}[\normalfont a)]
		\item If $x \leq y$ then $\s x \geq \s y$. \label{190326_02}
		\item If $x \leq y$ then $z \ra x \leq z \ra y$.   \label{180326_01}
		\item If $x \leq y$ then $x \ra z \geq y \ra z$.  \label{190326_04}
		\item $1 \ra x \leq x$.   \label{190326_03}
		\item If $x \to y = 1$ and $y \to z = 1$ then $x \to z = 1$.   \label{250426_12}
		\item $x \ra x \approx 1$.  \label{190326_05}
		\item If $x \leq y$ then $x \ra y = 1$.  \label{300326_01}
		\item If $x \to y = 1$ and $\s y \to \s x = 1$ then $x \leq y$.  \label{210426_01}
		
		\item $x \wedge x^* \leq \sim x$.  \label{250426_10}
		\item  $\sim x \wedge x \approx \sim x \wedge (1 \to x)$.   \label{250426_08}

		\item $(x \wedge x^*) \to y \approx 1$.   \label{250426_05}

		\item $[(x \wedge y) \to x^*] \to [(x \wedge y) \to z] \approx 1$.   \label{250426_16}
		\item $(x \wedge \sim (x^* \to y)) \to z \approx 1$.   \label{250426_13} 
		\item $(x \wedge \sim x^*) \to y \approx x \to y$.    \label{250426_17} 
		\item $x \to (\sim x^*) \approx 1$.    \label{270426_01} 
		\item $\sim (x \to y) \leq (x \to y) \vee x$.   \label{290426_01}

		\item $x \wedge (x \to \sim x) \approx x \wedge \sim x$. 	\label{270426_11} 
		
		\item $(x \wedge \sim x) \to y \approx 1$.    \label{270426_16}
		\item $(x\we \sim x)\we (y \vee \sim y) \approx x\we \sim x$.   \label{270426_18}

		\item  $(x \wedge \s y) \ra \s (x \ra y) \approx 1$.   \label{190326_01}

	\end{enumerate}
\end{lemma}

\begin{proof}
	\begin{itemize}
		\item[(\ref{190326_02})]	Since $x \leq y$ then $x \wedge y = x$,
		so $\s x \vee \s y = \s (x \wedge y)$ by (R\ref{negsup}). Then $\s x \vee \s y =  \s x$.

		\item[(\ref{180326_01})] Since $x \leq y$ then $x = x \we y$. Hence,
		%$$
		%\begin{array}{lcll}
		%z \ra x	& = & z \ra (x \wedge y) & \mbox{} \\
		%	& = & (z \ra x) \wedge (z \ra y) & \mbox{by } (R \ref{infimo}) \\
		%	& \leq  & z \ra y & \mbox{} 
		%\end{array}
		%$$	
		$z \ra x	$
		$\overset{  
		}{=}  z \ra (x \wedge y) $
		$\overset{  (R \ref{infimo}) 
		}{=}  (z \ra x) \wedge (z \ra y) $
		$\overset{  
		}{\leq}  z \ra y $
		
		\item[(\ref{190326_04})]
		%	
		%$$
		%\begin{array}{lcll}
		%y \ra z	& = & (x \vee y) \ra z & \mbox{} \\
		%	& = & (x \ra z) \wedge (y \ra z) & \mbox{by } (R \ref{supremo}) 
		%\end{array}
		%$$
		\noindent $y \ra z	$
		$\overset{  
		}{=}  (x \vee y) \ra z $
		$\overset{  (R \ref{supremo}) 
		}{=}  (x \ra z) \wedge (y \ra z) $
		
		\item[(\ref{190326_03})]	
		%$$
		%\begin{array}{lcll}
		%	1 \ra x	& = & 1 \wedge (1 \ra x) & \mbox{} \\
		%	& \leq & 1 \wedge (\s 1 \vee x) & \mbox{by (R\ref{inf})}  \\
		%	& = & x & \mbox{} 
		%\end{array}
		%$$
		\noindent $	1 \ra x	$
		$\overset{  
		}{=}  1 \wedge (1 \ra x) $
		$\overset{  (R\ref{inf})  
		}{\leq}  1 \wedge (\s 1 \vee x) $
		$\overset{  
		}{=}  x $

		\item[(\ref{250426_12})]
		
		%$$
		%\begin{array}{lcll}
		%1 \to (a \to c)	& = & (1 \wedge 1) \to (a \to c) & \mbox{} \\
		%	& = & ((a \to b) \wedge (b \to c)) \to (a \to c) & \mbox{} \\
		%	& = & 1 & \mbox{por (R\ref{trans})} 
		%\end{array}
		%$$
		\noindent $1 \to (a \to c)	$
		$\overset{  
		}{=}  (1 \wedge 1) \to (a \to c) $
		$\overset{  
		}{=}  ((a \to b) \wedge (b \to c)) \to (a \to c) $
		$\overset{ (R\ref{trans}) 
		}{=}  1 $.
		
		Thus, $a \to c = 1$, by item (\ref{190326_03}).

		\item[(\ref{190326_05})]
		
		Taking into account (\ref{190326_03}) we get $1 \ra (\s x) \leq \s x$. Then by item (\ref{190326_02}), $\s (1 \ra (\s x)) \geq (\s \s x$. Now, applying item (\ref{190326_04}), we obtain that $\s (1 \ra (\s x)) \ra (1 \wedge (\s \s x)) \leq (\s \s x) \ra (1 \wedge (\s \s x))$. By condition (R\ref{au}) we deduce that $1 \leq (\s \s x) \ra (1 \wedge (\s \s x))$. Thus, from (R\ref{dobleneg}), $x \ra x = 1$.

		\item[(\ref{300326_01})]
		
		%$$
		%\begin{array}{lcll}
		%x \ra y	& = & (x \ra y) \wedge 1 & \mbox{} \\
		%	& = & (x \ra y) \wedge (x \ra x)  & \mbox{by } (\ref{190326_05}) \\
		%	& = & x \ra (x \wedge y) & \mbox{by } (R \ref{infimo}) \\
		%	& = & x \ra x & \mbox{since } x \leq y \\
		%	& = & 1 & \mbox{by } (\ref{190326_05}) 
		%\end{array}
		%$$
		\noindent $x \ra y	$
		$\overset{  
		}{=}  (x \ra y) \wedge 1 $
		$\overset{  (\ref{190326_05}) 
		}{=}  (x \ra y) \wedge (x \ra x)  $
		$\overset{  (R \ref{infimo}) 
		}{=}  x \ra (x \wedge y) $
		$\overset{  x \leq y 
		}{=}  x \ra x $
		$\overset{  (\ref{190326_05}) 
		}{=}  1 $
		
		\item[(\ref{210426_01})] Since $\sim a \leq \sim a \vee a$ then by item (\ref{300326_01}) we have that
		\begin{equation} \label{210426_02}
			(\sim a) \to (\sim a \vee a) = 1.
		\end{equation}
		Also, by considering that that $\sim a \leq \sim a \vee a$ and item (\ref{180326_01}) we have that $1 = \sim b \to \sim a \leq \sim b \to (\sim a \vee a)$. Furthermore,
		\begin{equation} \label{210426_03}
			\sim b \to (\sim a \vee a) = 1.
		\end{equation}
		
		Then
		
		$$
		\begin{array}{lcll}
			\sim a \vee \sim b	& = & (\sim a \vee \sim b) \wedge ((\sim a) \to (\sim a \vee a)) \wedge (\sim b \to (\sim a \vee a) ) & \mbox{by (\ref{210426_02}) and (\ref{210426_03})} \\
			& = &  (\sim a \vee \sim b) \wedge [ (\sim a \vee \sim b) \to (\sim a \vee a) ] & \mbox{by (R\ref{supremo})} \\
			& \leq & (\sim a \vee \sim b) \wedge [ \sim (\sim a \vee \sim b) \vee (\sim a \vee a) ] & \mbox{by (R\ref{inf})} \\
			& = & (\sim a \vee \sim b) \wedge [ \sim \sim (a \wedge b) \vee (\sim a \vee a) ] & \mbox{by (R\ref{negsup})} \\
			& = & (\sim a \vee \sim b) \wedge [(a \wedge b) \vee (\sim a \vee a) ] & \mbox{by (R\ref{dobleneg})} \\
			& = & (\sim a \vee \sim b) \wedge (\sim a \vee a) & \mbox{since } a \wedge b \leq \sim a \vee a \\
			& = & \sim a \vee (\sim b \wedge a) & \mbox{} 
		\end{array}
		$$
		Then
		\begin{equation} \label{210426_04}
			\sim a \vee \sim b \leq \sim a \vee (\sim b \wedge a).
		\end{equation}
		Besides, recall that
		$$
		\begin{array}{lcll}
			a	& = & a \wedge 1 & \mbox{} \\
			& = & a \wedge (a \to b) & \mbox{} \\
			& \leq & a \wedge (\sim a \vee b) & \mbox{by (R\ref{inf})} \\
			& \leq & \sim a \vee b. & \mbox{} 
		\end{array}
		$$
		Then
		$$
		\begin{array}{lcll}
			\sim a	&  \geq & \sim (\sim a \vee b) & \mbox{by item (\ref{190326_02})} \\
			& = & \sim (\sim a \vee \sim \sim  b) & \mbox{by (R\ref{dobleneg})} \\
			& = & \sim \sim (a \wedge \sim b) & \mbox{by (R\ref{negsup})} \\
			& = & a \wedge \sim b. & \mbox{by (R\ref{dobleneg})} 
		\end{array}
		$$
		Hence,
		\begin{equation} \label{210426_05}
			\sim a \geq a \wedge \sim b.
		\end{equation}
		Moreover
		$$
		\begin{array}{lcll}
			\sim b	& \leq & \sim a \vee \sim b & \mbox{} \\
			& \leq  & \sim a \vee (\sim b \wedge a) & \mbox{by (\ref{210426_04})} \\
			& = & \sim a & \mbox{by (\ref{210426_05})} 
		\end{array}
		$$
		By item (\ref{190326_02}) we have that $\sim \sim b \geq \sim \sim a$. 
		Then, applying (R\ref{dobleneg}) we get $b \geq a$.
		
		\item[(\ref{250426_10})] 
		It follows from (R\ref{inf}) that $b \wedge b^* = b \wedge (b \to 0) \leq b \wedge \sim b \leq \sim b$.

		\item[(\ref{250426_08})]
		
		Note that
		$$
		\begin{array}{lcll}
			\sim (1 \to a)	& = & (\sim (1 \to a)) \wedge 1 & \mbox{} \\
			& = & \sim (1 \to a) \wedge (\sim (1 \to a)  \to (1 \wedge \sim a))  & \mbox{by (R\ref{au})} \\
			& \leq  & \sim (1 \to a) \wedge (\sim \sim (1 \to a)  \vee (1 \wedge \sim a))  & \mbox{by (R\ref{inf})} \\
			& = & \sim (1 \to a) \wedge ((1 \to a)  \vee (1 \wedge \sim a)) & \mbox{by (R\ref{dobleneg})} \\
			& = & \sim (1 \to a) \wedge ((1 \to a)  \vee \sim a) & \mbox{} \\
			& \leq & \sim (1 \to a) \wedge (a  \vee \sim a) & \mbox{by (\ref{190326_03})} \\
			& \leq & a  \vee \sim a. & \mbox{} 
		\end{array}
		$$
		Then the identity
		\begin{equation} \label{250426_07}
			\sim (1 \to x) \leq x  \vee \sim x
		\end{equation}
		is satisfied in $\mathbf T$. 
		Thus,
		%$$
		%\begin{array}{lcll}
		%	\sim a \wedge a	& = & \sim a \wedge a \wedge [(1 \to a) \vee \sim a] & \mbox{} \\
		%	& = & \sim a \wedge [(a \vee (1 \to a)) \wedge ((1 \to a) \vee \sim a)] & \mbox{por (\ref{190326_03})} \\
		%	& = & \sim a \wedge [(1 \to a) \vee (a \wedge \sim a)] & \mbox{} \\
		%	& = & \sim a \wedge \sim \sim [\sim \sim (1 \to a) \vee \sim \sim (a \wedge \sim a)] & \mbox{by (R\ref{dobleneg})} \\
		%	& = & \sim a \wedge \sim \sim \sim [\sim (1 \to a) \wedge \sim (a \wedge \sim a)] & \mbox{por (R\ref{negsup})} \\
		%	& = & \sim a \wedge \sim [\sim (1 \to a) \wedge \sim (a \wedge \sim a)] & \mbox{by (R\ref{dobleneg})} \\
		%	& = & \sim a \wedge \sim [\sim (1 \to a) \wedge (\sim a \vee \sim \sim a)] & \mbox{por (R\ref{negsup})} \\
		%	& = & \sim a \wedge \sim [\sim (1 \to a) \wedge (\sim a \vee a)] & \mbox{by (R\ref{dobleneg})} \\
		%	& = & \sim a \wedge \sim \sim (1 \to a) & \mbox{by (\ref{250426_07})} \\
		%	& = & \sim a \wedge (1 \to a) & \mbox{by (R\ref{dobleneg}).} 
		%\end{array}
		%$$
		\noindent $	\sim a \wedge a	$
		$\overset{  
		}{=}  \sim a \wedge a \wedge [(1 \to a) \vee \sim a] $
		$\overset{  (\ref{190326_03}) 
		}{=}  \sim a \wedge [(a \vee (1 \to a)) \wedge ((1 \to a) \vee \sim a)] $
		$\overset{  
		}{=}  \sim a \wedge [(1 \to a) \vee (a \wedge \sim a)] $
		$\overset{  (R\ref{dobleneg}) 
		}{=}  \sim a \wedge \sim \sim [\sim \sim (1 \to a) \vee \sim \sim (a \wedge \sim a)] $
		$\overset{  (R\ref{negsup}) 
		}{=}  \sim a \wedge \sim \sim \sim [\sim (1 \to a) \wedge \sim (a \wedge \sim a)] $
		$\overset{  (R\ref{dobleneg}) 
		}{=}  \sim a \wedge \sim [\sim (1 \to a) \wedge \sim (a \wedge \sim a)] $
		$\overset{  (R\ref{negsup}) 
		}{=}  \sim a \wedge \sim [\sim (1 \to a) \wedge (\sim a \vee \sim \sim a)] $
		$\overset{  (R\ref{dobleneg}) 
		}{=}  \sim a \wedge \sim [\sim (1 \to a) \wedge (\sim a \vee a)] $
		$\overset{  (\ref{250426_07}) 
		}{=}  \sim a \wedge \sim \sim (1 \to a) $
		$\overset{  (R\ref{dobleneg}). 
		}{=}  \sim a \wedge (1 \to a) $. 
		Thus, the identity
		$ %\label{250426_08}
		\sim x \wedge x \approx \sim x \wedge (1 \to x)
		$
		is satisfied in $\mathbf T$.
		
		\item[(\ref{250426_05})] 
		It follows from (R\ref{trans}) that $((1 \to a) \wedge (a \to 0)) \to (1 \to 0) = 1$. Notice that, using (\ref{190326_03}), we get $1 \to 0 = 0$. Then we obtain $((1 \to a) \wedge (a \to 0)) \to 0 = 1$. That is, $((1 \to a) \wedge a^*) \to 0 = 1$. Also, by item (\ref{300326_01}), we get $0 \to b =1$. Then
		$$
		\begin{array}{lcll}
			((1 \to a) \wedge a^*) \to b	& \geq  & 1 \to (((1 \to a) \wedge a^*) \to b) & \mbox{by (\ref{190326_03})} \\
			& \geq & (1 \wedge 1) \to (((1 \to a) \wedge a^*) \to b) & \mbox{} \\
			& \geq & ((((1 \to a) \wedge a^*) \to 0) \wedge (0 \to b)) \to (((1 \to a) \wedge a^*) \to b) & \mbox{} \\
			& \geq & 1 & \mbox{by (R\ref{trans}).} 
		\end{array}
		$$
		Hence, the identity
		\begin{equation} \label{250426_06}
			((1 \to x) \wedge x^*) \to y \approx 1
		\end{equation}
		is satisfied in $\mathbf T$.
		
		Then
		%$$
		%\begin{array}{lcll}
		%(a \wedge a^*) \to b	& = & (a \wedge a^* \wedge \sim a) \to b & \mbox{por (\ref{250426_10})} \\
		%	& = & ((1 \to a) \wedge a^* \wedge \sim a) \to b & \mbox{por (\ref{250426_08})} \\
		%	& \geq & 1 \to [((1 \to a) \wedge a^* \wedge \sim a) \to b] & \mbox{por (\ref{190326_03})} \\
		%	& = & (1 \wedge 1) \to [((1 \to a) \wedge a^* \wedge \sim a) \to b] & \mbox{} \\
		%	& = & [1 \wedge [((1 \to a) \wedge a^*) \to b]] \to [((1 \to a) \wedge a^* \wedge \sim a) \to b] & \mbox{por (\ref{250426_06})} \\
		%	& = & [[[(1 \to a) \wedge a^* \wedge \sim a] \to [(1 \to a) \wedge a^*]] \wedge [((1 \to a) \wedge a^*) \to b]] \to [((1 \to a) \wedge a^* \wedge \sim a) \to b] & \mbox{by (\ref{300326_01})} \\
		%	& = & 1 & \mbox{by (R\ref{inf})} 
		%\end{array}
		%$$
		%
		%$$
		%\begin{array}{cll}
		% & (a \wedge a^*) \to b &  \mbox{} \\
		%= & (a \wedge a^* \wedge \sim a) \to b &  \mbox{por (\ref{250426_10})} \\
		%= & ((1 \to a) \wedge a^* \wedge \sim a) \to b &  \mbox{por (\ref{250426_08})} \\
		%\geq & 1 \to [((1 \to a) \wedge a^* \wedge \sim a) \to b] &  \mbox{por (\ref{190326_03})} \\
		%= & (1 \wedge 1) \to [((1 \to a) \wedge a^* \wedge \sim a) \to b] &  \mbox{} \\
		%= & [1 \wedge [((1 \to a) \wedge a^*) \to b]] \to [((1 \to a) \wedge a^* \wedge \sim a) \to b] &  \mbox{por (\ref{250426_06})} \\
		%= & [[[(1 \to a) \wedge a^* \wedge \sim a] \to [(1 \to a) \wedge a^*]] \wedge [((1 \to a) \wedge a^*) \to b]] \to [((1 \to a) \wedge a^* \wedge \sim a) \to b] &  \mbox{by (\ref{300326_01})}  \\
		%= & 1 &  \mbox{by (R\ref{inf})} 
		%\end{array}
		%$$
		\noindent $	(a \wedge a^*) \to b	$
		$\overset{  (\ref{250426_10}) 
		}{=}  (a \wedge a^* \wedge \sim a) \to b $
		$\overset{  (\ref{250426_08}) 
		}{=}  ((1 \to a) \wedge a^* \wedge \sim a) \to b $
		$\overset{  (\ref{190326_03}) 
		}{\geq}  1 \to [((1 \to a) \wedge a^* \wedge \sim a) \to b] $
		$\overset{  
		}{=}  (1 \wedge 1) \to [((1 \to a) \wedge a^* \wedge \sim a) \to b] $
		$\overset{  (\ref{250426_06}) 
		}{=}  [1 \wedge [((1 \to a) \wedge a^*) \to b]] \to [((1 \to a) \wedge a^* \wedge \sim a) \to b] $
		$\overset{  (\ref{300326_01}) 
		}{=}  [[[(1 \to a) \wedge a^* \wedge \sim a] \to [(1 \to a) \wedge a^*]] \wedge [((1 \to a) \wedge a^*) \to b]] \to [((1 \to a) \wedge a^*
		\wedge \sim a) \to b] $
		$\overset{  (R\ref{inf}) 
		}{=}  1. $

		\item[(\ref{250426_16})]
		
		Note that
		%$$
		%\begin{array}{lcll}
		%	[(a \wedge b) \to a^*] \to [(a \wedge b) \to c]	& = & [((a \wedge b) \to a^*) \wedge 1] \to [(a \wedge b) \to c] & \mbox{} \\
		%	& = & [((a \wedge b) \to a^*) \wedge 1] \to [(a \wedge b) \to c] & \mbox{} \\
		%	& = & [((a \wedge b) \to a^*) \wedge ((a \wedge b) \to a)] \to [(a \wedge b) \to c] & \mbox{por (\ref{300326_01})} \\
		%	& = & [(a \wedge b) \to (a^* \wedge a)] \to [(a \wedge b) \to c] & \mbox{by  (R\ref{infimo})} \\
		%	& = & [((a \wedge b) \to (a^* \wedge a)) \wedge 1] \to [(a \wedge b) \to c] & \mbox{} \\
		%	& = & [((a \wedge b) \to (a^* \wedge a)) \wedge ((a^* \wedge a) \to c)] \to [(a \wedge b) \to c] & \mbox{by (\ref{250426_05})} \\
		%	& = & 1 & \mbox{by (R\ref{trans})} 
		%\end{array}
		%$$
		\noindent $	[(a \wedge b) \to a^*] \to [(a \wedge b) \to c]	$
		$\overset{  
		}{=}  [((a \wedge b) \to a^*) \wedge 1] \to [(a \wedge b) \to c] $
		$\overset{  
		}{=}  [((a \wedge b) \to a^*) \wedge 1] \to [(a \wedge b) \to c] $
		$\overset{  (\ref{300326_01}) 
		}{=}  [((a \wedge b) \to a^*) \wedge ((a \wedge b) \to a)] \to [(a \wedge b) \to c] $
		$\overset{   (R\ref{infimo}) 
		}{=}  [(a \wedge b) \to (a^* \wedge a)] \to [(a \wedge b) \to c] $
		$\overset{  
		}{=}  [((a \wedge b) \to (a^* \wedge a)) \wedge 1] \to [(a \wedge b) \to c] $
		$\overset{  (\ref{250426_05}) 
		}{=}  [((a \wedge b) \to (a^* \wedge a)) \wedge ((a^* \wedge a) \to c)] \to [(a \wedge b) \to c] $
		$\overset{  (R\ref{trans}) 
		}{=}  1. $ 
		Then the identity
		$$%\begin{equation} \label{250426_16}
		[(x \wedge y) \to x^*] \to [(x \wedge y) \to z] \approx 1
		$$%\end{equation}
		is satisfied in $\mathbf T$.

		%\item[(\ref{250426_14})] 

		\item[(\ref{250426_13})] 
		
		It follows from (R\ref{au}) that $\sim (a \to b) \to (a \wedge \sim b) = 1$. Also, by item  (\ref{300326_01}) we get $(a \wedge \sim b) \to a = 1$. Then, taking into account (\ref{250426_12}) we get $\sim (a \to b) \to a = 1$. Thus, the identity
		\begin{equation} \label{250426_14}  
			\sim (x \to y) \to x  \approx 1
		\end{equation}
		 is satisfied in $\mathbf T$.

		Recall that by (\ref{300326_01}) we have that $[a \wedge \sim (b \to c)] \to \sim (b \to c) = 1$. Consequently, by (\ref{250426_14}) and (\ref{250426_12}) we conclude that the identity
		\begin{equation} \label{250426_15}
			[x \wedge \sim (y \to z)] \to y \approx 1
		\end{equation}
		is satisfied in $\mathbf T$.
		%Siguiendo el siguiente cálculo
		%$$
		%\begin{array}{lcll}
		%[(a \wedge b) \to a^*] \to [(a \wedge b) \to c]	& = & [((a \wedge b) \to a^*) \wedge 1] \to [(a \wedge b) \to c] & \mbox{} \\
		%	& = & [((a \wedge b) \to a^*) \wedge 1] \to [(a \wedge b) \to c] & \mbox{} \\
		%	& = & [((a \wedge b) \to a^*) \wedge ((a \wedge b) \to a)] \to [(a \wedge b) \to c] & \mbox{por (\ref{300326_01})} \\
		%	& = & [(a \wedge b) \to (a^* \wedge a)] \to [(a \wedge b) \to c] & \mbox{by  (R\ref{infimo})} \\
		%	& = & [((a \wedge b) \to (a^* \wedge a)) \wedge 1] \to [(a \wedge b) \to c] & \mbox{} \\
		%	& = & [((a \wedge b) \to (a^* \wedge a)) \wedge ((a^* \wedge a) \to c)] \to [(a \wedge b) \to c] & \mbox{by (\ref{250426_05})} \\
		%	& = & 1 & \mbox{by (R\ref{trans})} 
		%\end{array}
		%$$
		%se concluye que el álgebra satisface la identidad
		%\begin{equation} \label{250426_16}
		%[(x \wedge y) \to x^*] \to [(x \wedge y) \to z] \approx 1.
		%\end{equation}
		Finally, note that
		%$$
		%\begin{array}{lcll}
		%1	& = & [(a \wedge \sim (a^* \to b)) \to a^*] \to [(a \wedge \sim (a^* \to b)) \to c] & \mbox{by (\ref{250426_16})} \\
		%	& = & 1 \to [(a \wedge \sim (a^* \to b)) \to c] & \mbox{by (\ref{250426_15})} \\
		%	& \leq  & (a \wedge \sim (a^* \to b)) \to c & \mbox{by (\ref{190326_03})} 
		%\end{array}
		%$$
		\noindent $1	$
		$\overset{  (\ref{250426_16}) 
		}{=}  [(a \wedge \sim (a^* \to b)) \to a^*] \to [(a \wedge \sim (a^* \to b)) \to c] $
		$\overset{  (\ref{250426_15}) 
		}{=}  1 \to [(a \wedge \sim (a^* \to b)) \to c] $
		$\overset{  (\ref{190326_03}) 
		}{\leq}  (a \wedge \sim (a^* \to b)) \to c $.
		
		\item[(\ref{250426_17})] 
		
		It follows from item (\ref{250426_10}) that 
		$b \wedge b^* =  \leq \sim b$.  Also, it follows from item (\ref{190326_02}) that
		\begin{equation} \label{250426_02}
		b \wedge \sim (b \wedge b^*) = b.
		\end{equation}
		Then
		%$$
		%\begin{array}{lcll}
		%	(a \wedge b) \to (\sim b \vee \sim b^*)	& = & (a \wedge b \wedge \sim (b \wedge b^*)) \to (\sim b \vee \sim b^*) & \mbox{by (\ref{250426_02})}  \\
		%	& = & (a \wedge b \wedge (\sim b \vee \sim b^*)) \to (\sim b \vee \sim b^*) & \mbox{by (R\ref{negsup})} \\
		%	& = & 1 & \mbox{by (\ref{300326_01})} 
		%\end{array}
		%$$
		\noindent $	(a \wedge b) \to (\sim b \vee \sim b^*)	$
		$\overset{  (\ref{250426_02})  
		}{=}  (a \wedge b \wedge \sim (b \wedge b^*)) \to (\sim b \vee \sim b^*) $
		$\overset{  (R\ref{negsup}) 
		}{=}  (a \wedge b \wedge (\sim b \vee \sim b^*)) \to (\sim b \vee \sim b^*) $
		$\overset{  (\ref{300326_01}) 
		}{=}  1 $. 
		
		Thus, the identity
		\begin{equation} \label{250426_03}
			(x \wedge y) \to (\sim y \vee \sim y^*) \approx 1
		\end{equation}
		is satisfied in $\mathbf T$.
		Also note that
		$$
		\begin{array}{ll}
			a \vee b \geq b	&  \mbox{} \\
			\sim (a \vee b) \leq \sim b	&  \mbox{by item (\ref{190326_02})} \\
			\sim (a \vee b) \vee \sim (a^* \wedge b^*) \leq \sim b	\vee \sim (a^* \wedge b^*) 	&  \mbox{} \\
			a \to (\sim (a \vee b) \vee \sim (a^* \wedge b^*)) \leq a \to (\sim b	\vee \sim (a^* \wedge b^*))	&  \mbox{by item (\ref{180326_01})} \\
			a \to (\sim (a \vee b) \vee \sim ((a \to 0) \wedge (b \to 0))) \leq a \to (\sim b	\vee \sim (a^* \wedge b^*))	&  \mbox{} \\
			a \to (\sim (a \vee b) \vee \sim ((a \vee b) \to 0)) \leq a \to (\sim b	\vee \sim (a^* \wedge b^*))	&  \mbox{by (R\ref{supremo})} \\
			a \to (\sim (a \vee b) \vee \sim (a \vee b)^*) \leq a \to (\sim b	\vee \sim (a^* \wedge b^*))	&  \mbox{} \\
			(a \wedge (a \vee b)) \to (\sim (a \vee b) \vee \sim (a \vee b)^*) \leq a \to (\sim b	\vee \sim (a^* \wedge b^*))	&  \mbox{} \\
			1 \leq a \to (\sim b	\vee \sim (a^* \wedge b^*))	&  \mbox{by (\ref{250426_03})} \\
			1 \leq a \to (\sim b	\vee \sim a^* \vee \sim b^*)	&  \mbox{by (R\ref{negsup}).} 
		\end{array}
		$$
		Hence, the identity
		\begin{equation} \label{250426_04}
			x \to (\sim y	\vee \sim x^* \vee \sim y^*) \approx 1.
		\end{equation}
		is satisfied in $\mathbf T$. 
		Taking into account (\ref{250426_04}) we get $a \to (\sim a^* \vee \sim a^{**}) = a \to (\sim a^* \vee \sim a^* \vee \sim a^{**}) = 1$.
		
		Then, the identity 
		\begin{equation} \label{250426_01}  
			x \to (\sim x^* \vee \sim x^{**}) \approx 1
		\end{equation}
		is satisfied in $\mathbf T$.

		Note that it is enough to show that $a \leq \sim a^* \vee \sim a^{**}$. 
		It follows from (\ref{250426_01}) that $$a \to (\sim a^* \vee \sim a^{**}) = 1. $$
		Besides,
		%$$
		%\begin{array}{lcll}
		%	[\sim (\sim a^* \vee \sim a^{**})] \to \sim a	& = & [\sim \sim (a^* \wedge a^{**})] \to \sim a & \mbox{by (R\ref{negsup})} \\
		%	& = & (a^* \wedge a^{**}) \to a & \mbox{by (R\ref{dobleneg})} \\
		%	& = & 1 & \mbox{by (\ref{250426_05})} 
		%\end{array}
		%$$
		\noindent $	[\sim (\sim a^* \vee \sim a^{**})] \to \sim a	$
		$\overset{  (R\ref{negsup}) 
		}{=}  [\sim \sim (a^* \wedge a^{**})] \to \sim a $
		$\overset{  (R\ref{dobleneg}) 
		}{=}  (a^* \wedge a^{**}) \to a $
		$\overset{  (\ref{250426_05}) 
		}{=}  1 $. 
		Then, by item (\ref{210426_01}), $a \leq \sim a^* \vee \sim a^{**}$.
		
		Hence, the identity
		\begin{equation} \label{250426_11}  
			x \wedge (\sim x^* \vee \sim x^{**}) \approx x
		\end{equation}
		is satisfied in $\mathbf T$.
		
		Then
		%$$
		%\begin{array}{lcll}
		%a \to b	& = & [a \wedge (\sim a^* \vee \sim a^{**})] \to b & \mbox{by (\ref{250426_11})} \\
		%	& = & [(a \wedge \sim a^*) \vee (a \wedge \sim a^{**})] \to b & \mbox{} \\
		%	& = & ((a \wedge \sim a^*) \to b) \wedge ((a \wedge \sim a^{**}) \to b) & \mbox{by (R\ref{supremo})} \\
		%	& = & ((a \wedge \sim a^*) \to b) \wedge 1& \mbox{by (\ref{250426_13})} \\
		%	& = & (a \wedge \sim a^*) \to b & \mbox{} 
		%\end{array}
		%$$
		\noindent $a \to b	$
		$\overset{  (\ref{250426_11}) 
		}{=}  [a \wedge (\sim a^* \vee \sim a^{**})] \to b $
		$\overset{  
		}{=}  [(a \wedge \sim a^*) \vee (a \wedge \sim a^{**})] \to b $
		$\overset{  (R\ref{supremo}) 
		}{=}  ((a \wedge \sim a^*) \to b) \wedge ((a \wedge \sim a^{**}) \to b) $
		$\overset{  (\ref{250426_13}) 
		}{=}  ((a \wedge \sim a^*) \to b) \wedge 1$
		$\overset{  
		}{=}  (a \wedge \sim a^*) \to b $.
		
		\item[(\ref{270426_01})] 
		$$
		\begin{array}{lcll}
			a \to (\sim a^*)	& = & (a \wedge \sim a^*) \to \sim a^* & \mbox{by (\ref{250426_17})} \\
			& = & 1 & \mbox{by (\ref{300326_01})} 
		\end{array}
		$$

		\item[(\ref{290426_01})] 
		
		By (\ref{250426_13}), $\sim (a \to b) \to a = 1$. Thus, 
		$$
		\begin{array}{lcll}
			\sim (a \to b)	& = & \sim (a \to b) \wedge 1 & \mbox{} \\
			& = & \sim (a \to b) \wedge (\sim (a \to b) \to a) & \mbox{} \\
			& \leq  & \sim (a \to b) \wedge (\sim \sim (a \to b) \vee a) & \mbox{by (R\ref{inf})} \\
			& = & \sim (a \to b) \wedge ((a \to b) \vee a) & \mbox{by (R\ref{dobleneg})} \\
			& \leq & (a \to b) \vee a & \mbox{} 
		\end{array}
		$$

		\item[(\ref{270426_11})] 
		Note that
		$$
		\begin{array}{lcll}
			a \wedge \sim a	& = & \sim \sim a \wedge \sim a & \mbox{by (R\ref{dobleneg})} \\
			& = & \sim \sim a \wedge (1 \to \sim a) & \mbox{by (\ref{250426_08})} \\
			& = &  a \wedge (1 \to \sim a) & \mbox{by (R\ref{dobleneg})} \\
			& \leq  & 1 \to \sim a & \mbox{} \\
			& = & 1 \wedge (1 \to \sim a) & \mbox{} \\
			& = & (a \to 1) \wedge (1 \to \sim a) & \mbox{by (\ref{300326_01})} 
		\end{array}
		$$
		It follows from (\ref{300326_01}) that
		$$%\begin{equation} \label{270426_12}
		(a \wedge \sim a) \to [(a \to 1) \wedge (1 \to \sim a)] = 1.
		$$%\end{equation}
		By (R\ref{trans}), $[(a \to 1) \wedge (1 \to \sim a)] \to (a \to \sim a) = 1$.
		Then by (\ref{250426_12}) we get
		\begin{equation} \label{270426_13}
			(a \wedge \sim a) \to (a \to \sim a) = 1.
		\end{equation}
		Applying (R\ref{au}) we get
		$1 = \sim (a \to \sim a) \to (a \wedge \sim \sim a) = \sim (a \to \sim a) \to  a$.
		As $a \leq  a \vee \sim a$, by (\ref{250426_12}) and (\ref{300326_01}),
		\begin{equation} \label{270426_14}
			\sim (a \to \sim a) \to (a \vee \sim a) = 1.
		\end{equation}
		Then
		%$$
		%\begin{array}{lcll}
		%1	& = & \sim (a \to \sim a) \to (a \vee \sim a) & \mbox{by (\ref{270426_14})} \\
		%	& = & \sim (a \to \sim a) \to (\sim \sim a \vee \sim a) & \mbox{by (R\ref{dobleneg})} \\
		%	& = & \sim (a \to \sim a) \to \sim (\sim a \wedge a) & \mbox{by (R\ref{negsup})} \\
		%	& = & \sim (a \to \sim a) \to (\sim a \wedge a) & \mbox{by (R\ref{dobleneg})} 
		%\end{array}
		%$$
		
		\noindent  $1	$
		$\overset{  (\ref{270426_14}) 
		}{=}  \sim (a \to \sim a) \to (a \vee \sim a) $
		$\overset{  (R\ref{dobleneg}) 
		}{=}  \sim (a \to \sim a) \to (\sim \sim a \vee \sim a) $
		$\overset{  (R\ref{negsup}) 
		}{=}  \sim (a \to \sim a) \to \sim (\sim a \wedge a) $.
		%$\overset{  (R\ref{dobleneg}) 
			%}{=}  \sim (a \to \sim a) \to (\sim a \wedge a) $.
		Hence,
		\begin{equation} \label{270426_15}
			\sim (a \to \sim a) \to \sim (\sim a \wedge a) = 1.
		\end{equation}
		
		Taking into account the equalities {\bf (\ref{270426_13})} and (\ref{270426_15}), and using item (\ref{210426_01}), we get $\sim a \wedge a \leq a \to \sim a$. Hence,
		$$\sim a \wedge a  \leq a \wedge (a \to \sim a).$$
		Besides,
		$$
		\begin{array}{lcll}
			a \wedge (a \to \sim a)	& \leq & a \wedge (\sim a \vee \sim a) & \mbox{by (R\ref{inf})} \\
			& = & a \wedge \sim a & \mbox{} 
		\end{array}
		$$
		Therefore, $a \wedge (a \to \sim a) =  a \wedge \sim a$.
		
		\item[(\ref{270426_16})]

		By item (\ref{290426_01}) we obtain that $\sim (a \to b) \leq (a \to b) \vee a$.
		Then
		$$
		\begin{array}{lcll}
			a \to b	& = & \sim \sim (a \to b) & \mbox{by (R\ref{dobleneg})} \\
			& \geq & \sim ((a \to b) \vee a) & \mbox{by (\ref{190326_02})} \\
			& = & \sim (\sim \sim (a \to b) \vee \sim \sim a) & \mbox{by (R\ref{dobleneg})} \\
			& = & \sim \sim ( \sim (a \to b) \wedge  \sim a) & \mbox{by (R\ref{negsup})} \\
			& = & \sim (a \to b) \wedge  \sim a & \mbox{by (R\ref{dobleneg})}
		\end{array}
		$$
		By (\ref{300326_01}), $(\sim (a \to b) \wedge  \sim a) \to (a \to b) = 1$.
		Thus,
		\begin{equation} \label{270426_10} 
			(\sim x \wedge \sim (x \to y)) \to (x \to y) \approx 1.
		\end{equation}

		Note that
		$$
		\begin{array}{lcll}
			a \wedge (a \to a^*)	& = & a \wedge (a \to a^*) \wedge (a \to a^*) & \mbox{} \\
			& \leq &  a \wedge (\sim a \vee a^*) \wedge (a \to a^*)  & \mbox{by (R\ref{inf})} \\
			& = & [a \wedge (a \to a^*) \wedge \sim a] \vee [a \wedge (a \to a^*) \wedge a^*] & \mbox{} \\
			& = & [a \wedge \sim a \wedge (a \to a^*)] \vee [a \wedge (a \to a^*) \wedge a^*] & \mbox{} \\
			& = & [(1 \to a) \wedge \sim a \wedge (a \to a^*)] \vee [a \wedge (a \to a^*) \wedge a^*] & \mbox{by (\ref{250426_08})} \\
			& \leq  & [(1 \to a) \wedge (a \to a^*)] \vee a^* & \mbox{} 
		\end{array}
		$$
		Then,
		\begin{equation} \label{270426_03}
			a \wedge (a \to a^*) \leq [(1 \to a) \wedge (a \to a^*)] \vee a^*
		\end{equation}

		Besides, by (R\ref{trans}) and item (\ref{300326_01}), we have that 
		$[(1 \to a) \wedge (a \to a^*)] \to (1 \to a^*) = 1$ and $(1 \to a^*) \to [(1 \to a^*) \vee a^*] = 1$. By (\ref{250426_12}), 
		$$%\begin{equation} \label{270426_04}
		[(1 \to a) \wedge (a \to a^*)] \to [(1 \to a^*) \vee a^*] = 1.
		$$%\end{equation}
		By (\ref{190326_03}), we get $(1 \to a^*) \vee a^* = a^*$. Thus,
		\begin{equation} \label{270426_05}
			[(1 \to a) \wedge (a \to a^*)] \to a^* = 1.
		\end{equation}
		Also
		%$$
		%\begin{array}{lcll}
		%	[[(1 \to a) \wedge (a \to a^*)] \vee a^*] \to a^*	& = & [[(1 \to a) \wedge (a \to a^*)] \to a^* ] \wedge [a^* \to a^*] & \mbox{by (R\ref{supremo})} \\
		%	& = & 1 \wedge [a^* \to a^*] & \mbox{by (\ref{270426_05})} \\
		%	& = & 1 & \mbox{by (\ref{190326_05})} 
		%\end{array}
		%$$
		\noindent $	[[(1 \to a) \wedge (a \to a^*)] \vee a^*] \to a^*	$
		$\overset{  (R\ref{supremo}) 
		}{=}  [[(1 \to a) \wedge (a \to a^*)] \to a^* ] \wedge [a^* \to a^*] $
		$\overset{  (\ref{270426_05}) 
		}{=}  1 \wedge [a^* \to a^*] $
		$\overset{  (\ref{190326_05}) 
		}{=}  1 $. 
		Now, from (\ref{270426_03}) and item (\ref{190326_04}), 
		$$%\begin{equation} \label{270426_06}
		[[(1 \to a) \wedge (a \to a^*)] \vee a^*] \to a^* \leq [a \wedge (a \to a^*)] \to a^*.
		$$%\end{equation}
		Then
		\begin{equation} \label{270426_07}
			[a \wedge (a \to a^*)] \to a^* = 1.
		\end{equation}
		Thus,
		$$
		\begin{array}{lcll}
			1	& = & [(a \wedge (a \to a^*)) \to a^*] \to [(a \wedge (a \to a^*)) \to b] & \mbox{by (\ref{250426_16})} \\
			& = & 1 \to [(a \wedge (a \to a^*)) \to b] & \mbox{by (\ref{270426_07})} \\
			& \leq & (a \wedge (a \to a^*)) \to b & \mbox{by (\ref{190326_03})} 
		\end{array}
		$$
		
		Furthermore,
		\begin{equation} \label{270426_02} 
			\mathbf T \models (x \wedge (x \to x^*)) \to y \approx 1.
		\end{equation}
		
		Notice that, by (R\ref{trans}) and item (\ref{300326_01}), we have that
		$[(a \to b) \wedge (b \to a^*)] \to (a \to a^*) = 1$ and $[a \wedge (a \to b) \wedge (b \to a^*)] \to [(a \to b) \wedge (b \to a^*)] = 1$. By (\ref{250426_12}), 
		$$%\begin{equation} \label{270426_04}
		[a \wedge (a \to b) \wedge (b \to a^*)] \to (a \to a^*) = 1.
		$$%\end{equation}
		From (\ref{300326_01}), $[a \wedge (a \to b) \wedge (b \to a^*)] \to a = 1$.
		Thus, by (R\ref{infimo}), we obtain that
		$[a \wedge (a \to b) \wedge (b \to a^*)] \to (a \wedge (a \to a^*)) = [[a \wedge (a \to b) \wedge (b \to a^*)] \to a] \wedge [[a \wedge (a \to b) \wedge (b \to a^*)] \to (a \to a^*)]  = 1$
		%$$
		%\begin{array}{lcll}
		%[a \wedge (a \to b) \wedge (b \to a^*)] \to (a \wedge (a \to a^*))	& = & [[a \wedge (a \to b) \wedge (b \to a^*)] \to a] \wedge [[a \wedge (a \to b) \wedge (b \to a^*)] \to (a \to a^*)] & \mbox{by (R\ref{infimo}) } \\
		%	& = &  & \mbox{} \\
		%	& = &  & \mbox{} \\
		%	& = &  & \mbox{} \\
		%	& = &  & \mbox{} \\
		%	& = &  & \mbox{} \\
		%	& = &  & \mbox{} \\
		%	& = &  & \mbox{} \\
		%	& = &  & \mbox{} \\
		%	& = &  & \mbox{} \\
		%	& = &  & \mbox{} \\
		%	& = &  & \mbox{} \\
		%	& = & &
		%\end{array}
		%$$
		Then
		\begin{equation} \label{270426_08}
			[a \wedge (a \to b) \wedge (b \to a^*)] \to (a \wedge (a \to a^*)) = 1.
		\end{equation}
		The equality
		\begin{equation} \label{270426_09}
			(a \wedge (a \to a^*)) \to c = 1.
		\end{equation}
		is verified by  (\ref{270426_02}).
		By (\ref{250426_12}),  $[a \wedge (a \to b) \wedge (b \to a^*)]  \to c = 1$.
		Then, the identity 
		\begin{equation} \label{270426_06} 
			(x \wedge (x \to y) \wedge (y \to x^*)) \to z \approx 1
		\end{equation}
		is satisfied.
		
		Recall that
		%$$
		%\begin{array}{lcll}
		%a \wedge \sim a \wedge (a \to \sim a^*)	& = & a \wedge (a \to \sim a) \wedge (a \to \sim a^*) & \mbox{by (\ref{270426_11})} \\
		%	& = & a \wedge [a \to (\sim a \wedge \sim a^*)] & \mbox{by (R\ref{infimo})} \\
		%	& = & a \wedge [a \to (\sim a \wedge \sim a^*)] \wedge 1 & \mbox{} \\
		%	& = & a \wedge [a \to (\sim a \wedge \sim a^*)] \wedge [(\sim a \wedge \sim a^*) \to a^*] & \mbox{by (\ref{270426_10})} 
		%\end{array}
		%$$
		\noindent $a \wedge \sim a \wedge (a \to \sim a^*)	$
		$\overset{  (\ref{270426_11}) 
		}{=}  a \wedge (a \to \sim a) \wedge (a \to \sim a^*) $
		$\overset{  (R\ref{infimo}) 
		}{=}  a \wedge [a \to (\sim a \wedge \sim a^*)] $
		$\overset{  
		}{=}  a \wedge [a \to (\sim a \wedge \sim a^*)] \wedge 1 $
		$\overset{  (\ref{270426_10}) 
		}{=}  a \wedge [a \to (\sim a \wedge \sim a^*)] \wedge [(\sim a \wedge \sim a^*) \to a^*] $. 
		Moreover
		\begin{equation} \label{270426_17}
			a \wedge \sim a \wedge (a \to \sim a^*) = a \wedge [a \to (\sim a \wedge \sim a^*)] \wedge [(\sim a \wedge \sim a^*) \to a^*].
		\end{equation}
		Then,
		%$$
		%\begin{array}{lcll}
		%1	& = & [a \wedge (a \to (\sim a \wedge \sim a^*)) \wedge ((\sim a \wedge \sim a^*) \to a^*)] \to b & \mbox{by (\ref{270426_06})} \\
		%	& = & [a \wedge \sim a \wedge (a \to \sim a^*)] \to b & \mbox{by (\ref{270426_17})} \\
		%	& = & [a \wedge \sim a \wedge 1] \to b & \mbox{by (\ref{270426_01})} \\
		%	& = & [a \wedge \sim a] \to b & \mbox{} 
		%\end{array}
		%$$
		\noindent $1	$
		$\overset{  (\ref{270426_06}) 
		}{=}  [a \wedge (a \to (\sim a \wedge \sim a^*)) \wedge ((\sim a \wedge \sim a^*) \to a^*)] \to b $
		$\overset{  (\ref{270426_17}) 
		}{=}  [a \wedge \sim a \wedge (a \to \sim a^*)] \to b $
		$\overset{  (\ref{270426_01}) 
		}{=}  [a \wedge \sim a \wedge 1] \to b $
		$\overset{  
		}{=}  [a \wedge \sim a] \to b $.
		
		\item[(\ref{270426_18})] 
		By item (\ref{270426_16}), 
		\begin{equation} \label{270426_19}
			(a \wedge \sim a) \to (b \vee \sim b) = 1.
		\end{equation}
		Also, 
		%$$
		%\begin{array}{lcll}
		%\sim (b \vee \sim b) \to \sim (a \wedge \sim a)	& = & \sim (\sim \sim b \vee \sim b) \to \sim (a \wedge \sim a)	 & \mbox{by (R\ref{dobleneg})} \\
		%	& = & \sim \sim ( \sim b \wedge  b) \to \sim (a \wedge \sim a) & \mbox{by (R\ref{negsup})} \\
		%	& = & ( \sim b \wedge  b) \to \sim (a \wedge \sim a) & \mbox{by (R\ref{dobleneg})} \\
		%	& = & 1 & \mbox{by (\ref{270426_16})} 
		%\end{array}
		%$$
		\noindent $\sim (b \vee \sim b) \to \sim (a \wedge \sim a)	$
		$\overset{  (R\ref{dobleneg}) 
		}{=}  \sim (\sim \sim b \vee \sim b) \to \sim (a \wedge \sim a)	 $
		$\overset{  (R\ref{negsup}) 
		}{=}  \sim \sim ( \sim b \wedge  b) \to \sim (a \wedge \sim a) $
		$\overset{  (R\ref{dobleneg}) 
		}{=}  ( \sim b \wedge  b) \to \sim (a \wedge \sim a) $
		$\overset{  (\ref{270426_16}) 
		}{=}  1 $
		and, hence,
		\begin{equation} \label{270426_20}
			\sim (b \vee \sim b) \to \sim (a \wedge \sim a)	 = 1.
		\end{equation}
		Taking into account (\ref{270426_19}) and (\ref{270426_20}) we get
		$a \wedge \sim a \leq b \vee \sim b$ by item (\ref{210426_01}).

		\item[(\ref{190326_01})] 
		
		%$$
		%\begin{array}{lcll}
		%(a \to a^*) \to (a \to b)	& = & (a \to a^*) \to ((a \wedge \sim a^*) \to b) & \mbox{by (\ref{250426_17})} \\
		%	& = & [(a \to a^*) \wedge 1] \to ((a \wedge \sim a^*) \to b) & \mbox{} \\
		%	& = & [(a \to a^*) \wedge (a \to \sim a^*)] \to ((a \wedge \sim a^*) \to b) & \mbox{by (\ref{270426_01})} \\
		%	& = & [a \to (a^* \wedge \sim a^*)] \to ((a \wedge \sim a^*) \to b) & \mbox{by (R\ref{infimo})} \\
		%	& = & [[a \to (a^* \wedge \sim a^*)] \wedge 1] \to ((a \wedge \sim a^*) \to b)  & \mbox{} \\
		%	& = & [[a \to (a^* \wedge \sim a^*)] \wedge [(a^* \wedge \sim a^*) \to b]] \to ((a \wedge \sim a^*) \to b) & \mbox{by (\ref{270426_16})} \\
		%	& = & 1 & \mbox{by (R\ref{trans})} 
		%\end{array}
		%$$
		
		%\noindent $(a \to a^*) \to (a \to b)	$
		%$\overset{  (\ref{250426_17}) 
			%}{=}  (a \to a^*) \to ((a \wedge \sim a^*) \to b) $
		%$\overset{  
			%}{=}  [(a \to a^*) \wedge 1] \to ((a \wedge \sim a^*) \to b) $
		%$\overset{  (\ref{270426_01}) 
			%}{=}  [(a \to a^*) \wedge (a \to \sim a^*)] \to ((a \wedge \sim a^*) \to b) $
		%$\overset{  (R\ref{infimo}) 
			%}{=}  [a \to (a^* \wedge \sim a^*)] \to ((a \wedge \sim a^*) \to b) $
		%$\overset{  
			%}{=}  [[a \to (a^* \wedge \sim a^*)] \wedge 1] \to ((a \wedge \sim a^*) \to b)  $
		%$\overset{  (\ref{270426_16}) 
			%}{=}  [[a \to (a^* \wedge \sim a^*)] \wedge [(a^* \wedge \sim a^*) \to b]] \to ((a \wedge \sim a^*) \to b) $
		%$\overset{  (R\ref{trans}) 
			%}{=}  1 $. 

		\noindent $(a \to a^*) \to (a \to b)	$
		$\overset{  (\ref{250426_17}) 
		}{=}  (a \to a^*) \to ((a \wedge \sim a^*) \to b) $
		$\overset{  
		}{=}  [(a \to a^*) \wedge 1] \to ((a \wedge \sim a^*) \to b) $
		$\overset{  (\ref{270426_01}) 
		}{=}  [(a \to a^*) \wedge (a \to \sim a^*)] \to ((a \wedge \sim a^*) \to b) $
		$\overset{  (R\ref{infimo}) 
		}{=}  [a \to (a^* \wedge \sim a^*)] \to ((a \wedge \sim a^*) \to b) $
		$\overset{  
		}{=}  [[a \to (a^* \wedge \sim a^*)] \wedge 1] \to ((a \wedge \sim a^*) \to b)  $
		$\overset{  (\ref{270426_16}) 
		}{=}  [[a \to (a^* \wedge \sim a^*)] \wedge [(a^* \wedge \sim a^*) \to b]] \to ((a \wedge \sim a^*) \to b) $.
		Thus,
		$%\begin{equation} \label{090626_01}
		(a \to a^*) \to (a \to b) = [[a \to (a^* \wedge \sim a^*)] \wedge [(a^* \wedge \sim a^*) \to b]] \to ((a \wedge \sim a^*) \to b).
		$%\end{equation}

		%Por (R\ref{trans}), $[[a \to (a^* \wedge \sim a^*)] \wedge [(a^* \wedge \sim a^*) \to b]] \to (a \to b) = 1$. 
		Since $a \wedge \sim a^* \leq a$ then $(a \wedge \sim a^*) \to b \geq a \to b$ by (\ref{190326_04}). It follows from (\ref{180326_01}) that $[[a \to (a^* \wedge \sim a^*)] \wedge [(a^* \wedge \sim a^*) \to b]] \to ((a \wedge \sim a^*) \to b) \geq [[a \to (a^* \wedge \sim a^*)] \wedge [(a^* \wedge \sim a^*) \to b]] \to (a \to b) = 1$ by (R\ref{trans}).

		Then
		\begin{equation} \label{290426_128}
			\mathbf T \models (x \to x^*) \to (x \to y) \approx 1.
		\end{equation}
		
		Also, 
		%$$
		%\begin{array}{lcll}
		%\sim (a \to b) \to c	& = & [(\sim (a \to b)) \wedge [(a \to b) \vee a]] \to c & \mbox{by (\ref{290426_01})} \\
		%	& = & [[\sim (a \to b) \wedge (a \to b)] \vee [\sim (a \to b) \wedge a]] \to c & \mbox{} \\
		%	& = & [[\sim (a \to b) \wedge (a \to b)] \to c] \wedge [[\sim (a \to b) \wedge a] \to c] & \mbox{by (R\ref{supremo})} \\
		%	& = & [1 \wedge [[\sim (a \to b) \wedge a] \to c] & \mbox{by (\ref{270426_16})} \\
		%	& = & [\sim (a \to b) \wedge a] \to c & \mbox{} 
		%\end{array}
		%$$
		\noindent $\sim (a \to b) \to c	$
		$\overset{  (\ref{290426_01}) 
		}{=}  [(\sim (a \to b)) \wedge [(a \to b) \vee a]] \to c $
		$\overset{  
		}{=}  [[\sim (a \to b) \wedge (a \to b)] \vee [\sim (a \to b) \wedge a]] \to c $
		$\overset{  (R\ref{supremo}) 
		}{=}  [[\sim (a \to b) \wedge (a \to b)] \to c] \wedge [[\sim (a \to b) \wedge a] \to c] $
		$\overset{  (\ref{270426_16}) 
		}{=}  [1 \wedge [[\sim (a \to b) \wedge a] \to c] $
		$\overset{  
		}{=}  [\sim (a \to b) \wedge a] \to c $. 
		Then the equality 
		\begin{equation} \label{290426_156}
			\sim (x \to y) \to z \approx [x \wedge \sim (x \to y)] \to z.
		\end{equation}
		is satisfied.
		
		Note that
		$$
		\begin{array}{lcll}
			a \to b	& = & (a \wedge \sim a^*) \to b & \mbox{by (\ref{250426_17})} \\
			& = & (\sim a^*) \to b & \mbox{by (\ref{290426_156}),} 
		\end{array}
		$$
		Then
		\begin{equation} \label{290426_157}
			\mathbf T \models x \to y \approx \sim x^* \to y.
		\end{equation}
		Also note that
		%$$
		%\begin{array}{lcll}
		%[a \wedge (a^* \vee b) \wedge c] \to b	& = & [(a \wedge c \wedge a^*) \vee (a \wedge c \wedge b)] \to b & \mbox{} \\
		%	& = & [(a \wedge c \wedge a^*) \to b] \wedge [(a \wedge c \wedge b) \to b] & \mbox{by (R\ref{supremo})} \\
		%	& = & [(a \wedge c \wedge a^*) \to b] \wedge 1 & \mbox{by (\ref{300326_01})} \\
		%	& = & (a \wedge c \wedge a^*) \to b & \mbox{} 
		%\end{array}
		%$$
		\noindent $[a \wedge (a^* \vee b) \wedge c] \to b	$
		$\overset{  
		}{=}  [(a \wedge c \wedge a^*) \vee (a \wedge c \wedge b)] \to b $
		$\overset{  (R\ref{supremo}) 
		}{=}  [(a \wedge c \wedge a^*) \to b] \wedge [(a \wedge c \wedge b) \to b] $
		$\overset{  (\ref{300326_01}) 
		}{=}  [(a \wedge c \wedge a^*) \to b] \wedge 1 $
		$\overset{  
		}{=}  (a \wedge c \wedge a^*) \to b $. 
		Also, by (\ref{300326_01}), we have that $(a \wedge c \wedge a^*) \to (a \wedge a^*) = 1$, and, by (\ref{250426_05}), we have that $(a \wedge a^*) \to b = 1$. Hence, applying (\ref{250426_12}), we have that
		$(a \wedge c \wedge a^*) \to  b = 1$. Then
		\begin{equation} \label{290426_158}
			\mathbf T \models [x \wedge (x^* \vee y) \wedge z] \to y \approx 1.
		\end{equation}
		Besides,
		%$$
		%\begin{array}{lcll}
		%a \vee \sim (\sim a \to a)	& = & (\sim \sim a) \vee \sim (\sim a \to a) & \mbox{by (R\ref{dobleneg})} \\
		%	& = & \sim (\sim a \wedge (\sim a \to a)) & \mbox{by (R\ref{negsup})} \\
		%	& = & \sim (\sim a \wedge (\sim a \to \sim \sim a)) & \mbox{by (R\ref{dobleneg})} \\
		%	& = & \sim (\sim a \wedge \sim \sim a) & \mbox{by (\ref{270426_11})} \\
		%	& = & \sim (\sim a \wedge a) & \mbox{by (R\ref{dobleneg})} \\
		%	& = & (\sim \sim a) \vee (\sim a) & \mbox{by (R\ref{negsup})} \\
		%	& = & a \vee (\sim a) & \mbox{by (R\ref{dobleneg}).} 
		%\end{array}
		%$$
		\noindent $a \vee \sim (\sim a \to a)	$
		$\overset{  (R\ref{dobleneg}) 
		}{=}  (\sim \sim a) \vee \sim (\sim a \to a) $
		$\overset{  (R\ref{negsup}) 
		}{=}  \sim (\sim a \wedge (\sim a \to a)) $
		$\overset{  (R\ref{dobleneg}) 
		}{=}  \sim (\sim a \wedge (\sim a \to \sim \sim a)) $
		$\overset{  (\ref{270426_11}) 
		}{=}  \sim (\sim a \wedge \sim \sim a) $
		$\overset{  (R\ref{dobleneg}) 
		}{=}  \sim (\sim a \wedge a) $
		$\overset{  (R\ref{negsup}) 
		}{=}  (\sim \sim a) \vee (\sim a) $
		$\overset{  (R\ref{dobleneg}) 
		}{=}  a \vee (\sim a) $.
		
		Then
		\begin{equation} \label{290426_150}
			\mathbf T \models x \vee \sim (\sim x \to x) \approx x \vee (\sim x).
		\end{equation}
		
		By (\ref{300326_01}) and (\ref{270426_01}), $(a \wedge c) \to a = 1$, $a \to (\sim a^*) = 1$ and $\sim a^* \to ((a \to b) \vee \sim a^*) = 1$. Then it follows from (\ref{250426_12}) that  
		\begin{equation} \label{290426_02}
			(a \wedge c) \to ((a \to b) \vee \sim a^*) = 1.
		\end{equation}

		By (\ref{300326_01}), $(\sim (a \to b) \wedge a^*) \to a^* = 1$ and $(\sim (a \to b) \wedge a^*) \to (\sim (a \to b)) = 1.$
		
		Besides, by axiom (R\ref{au}), $\sim (a \to b) \to (a \wedge \sim b) = 1$. Then by (\ref{250426_12}), $(\sim (a \to b) \wedge a^*) \to (a \wedge \sim b) = 1$. Thus,
		$(\sim (a \to b) \wedge a^*) \to (a^* \wedge a \wedge \sim b) = [(\sim (a \to b) \wedge a^*) \to a^*] \wedge [(\sim (a \to b) \wedge a^*) \to (a \wedge \sim b)] = 1 \wedge 1 = 1$ by  (R\ref{infimo}). Then
		\begin{equation} \label{290426_03}
			(\sim (a \to b) \wedge a^*) \to (a^* \wedge a \wedge \sim b) = 1.
		\end{equation}
		By (\ref{300326_01}), $(a^* \wedge a \wedge \sim b) \to (a^* \wedge a) = 1$. Also, $(a^* \wedge a) \to \sim (a \wedge c) = 1$ by item (\ref{250426_05}). Taking into account that (\ref{250426_12}) we get
		$(a^* \wedge a \wedge \sim b) \to  \sim (a \wedge c) = 1$. Then, by (\ref{290426_03}) and, again, by item (\ref{250426_12}), we conclude that
		\begin{equation} \label{290426_04}
			(\sim (a \to b) \wedge a^*) \to \sim (a \wedge c) = 1.
		\end{equation}
		Thus,
		%$$
		%\begin{array}{lcll}
		%\sim ((a \to b) \vee \sim a^*) \to \sim (a \wedge c)	& = & \sim (\sim \sim (a \to b) \vee \sim a^*) \to \sim (a \wedge c) & \mbox{by (R\ref{dobleneg})} \\
		%	& = & [\sim \sim  (\sim (a \to b) \wedge a^*)] \to \sim (a \wedge c) & \mbox{by (R\ref{negsup})} \\
		%	& = & (\sim (a \to b) \wedge a^*) \to \sim (a \wedge c) & \mbox{by (R\ref{dobleneg})} \\
		%	& = & 1 & \mbox{by (\ref{290426_04})} 
		%\end{array}
		%$$
		\noindent $\sim ((a \to b) \vee \sim a^*) \to \sim (a \wedge c)	$
		$\overset{  (R\ref{dobleneg}) 
		}{=}  \sim (\sim \sim (a \to b) \vee \sim a^*) \to \sim (a \wedge c) $
		$\overset{  (R\ref{negsup}) 
		}{=}  [\sim \sim  (\sim (a \to b) \wedge a^*)] \to \sim (a \wedge c) $
		$\overset{  (R\ref{dobleneg}) 
		}{=}  (\sim (a \to b) \wedge a^*) \to \sim (a \wedge c) $
		$\overset{  (\ref{290426_04}) 
		}{=}  1 $, 
		that is,  
		\begin{equation} \label{290426_05}
			\sim ((a \to b) \vee \sim a^*) \to \sim (a \wedge c) = 1.
		\end{equation}
		
		It follows from (\ref{290426_02}), (\ref{290426_05}) and (\ref{210426_01}) that $a \wedge c \leq (a \to b) \vee \sim a^*$.
		Hence, the identity
		\begin{equation} \label{290426_149}
			x \wedge z \wedge ((x \to y) \vee \sim x^*) \approx x \wedge z.
		\end{equation}
		is satisfied in $\mathbf T$.
		
	  Besides,
		%$$
		%\begin{array}{lcll}
		%1	& = & [b \wedge (b^* \vee (\sim (b \to b^*))) \wedge a] \to \sim (b \to b^*) & \mbox{by (\ref{290426_158})} \\
		%	& = & [b \wedge (b^* \vee (\sim (\sim b^* \to b^*))) \wedge a] \to \sim (b \to b^*) & \mbox{by (\ref{290426_157})} \\
		%	& = & [b \wedge (b^* \vee \sim b^*) \wedge a] \to \sim (b \to b^*) & \mbox{by (\ref{290426_150})} \\
		%	& = & [b \wedge ((b \to 0) \vee \sim b^*) \wedge a] \to \sim (b \to b^*) & \mbox{} \\
		%	& = & (b \wedge a) \to \sim (b \to b^*) & \mbox{by (\ref{290426_149})} 
		%\end{array}
		%$$
		\noindent $1	$
		$\overset{  (\ref{290426_158}) 
		}{=}  [b \wedge (b^* \vee (\sim (b \to b^*))) \wedge a] \to \sim (b \to b^*) $
		$\overset{  (\ref{290426_157}) 
		}{=}  [b \wedge (b^* \vee (\sim (\sim b^* \to b^*))) \wedge a] \to \sim (b \to b^*) $
		$\overset{  (\ref{290426_150}) 
		}{=}  [b \wedge (b^* \vee \sim b^*) \wedge a] \to \sim (b \to b^*) $
		$\overset{  
		}{=}  [b \wedge ((b \to 0) \vee \sim b^*) \wedge a] \to \sim (b \to b^*) $
		$\overset{  (\ref{290426_149}) 
		}{=}  (b \wedge a) \to \sim (b \to b^*) $. 
		Then, the identity
		\begin{equation} \label{290426_160}
			(x \wedge y) \to \sim (y \to y^*) \approx 1.
		\end{equation}
		is satisfied in $\mathbf T$.
		
		Since
		%$$
		%\begin{array}{lcll}
		%\sim a^* \to \sim (a \to a^*)	& = & a \to \sim (a \to a^*) & \mbox{by (\ref{290426_157})} \\
		%	& = & 1 & \mbox{by (\ref{290426_160})} 
		%\end{array}
		%$$
		\noindent $\sim a^* \to \sim (a \to a^*)	$
		$\overset{  (\ref{290426_157}) 
		}{=}  a \to \sim (a \to a^*) $
		$\overset{  (\ref{290426_160}) 
		}{=}  1 $ 
		and
		%$$
		%\begin{array}{lcll}
		%\sim \sim (a \to a^*) \to \sim \sim a^*	& = & (a \to a^*) \to a^* & \mbox{by (R\ref{dobleneg})} \\
		%	& = & 1 & \mbox{by (\ref{290426_128})} 
		%\end{array}
		%$$
		\noindent $\sim \sim (a \to a^*) \to \sim \sim a^*	$
		$\overset{  (R\ref{dobleneg}) 
		}{=}  (a \to a^*) \to a^* $
		$\overset{  (\ref{290426_128}) 
		}{=}  1 $ 
		then, by (\ref{210426_01}), the identity
		\begin{equation} \label{290426_163}
			\sim x^* \leq \sim (x \to x^*)
		\end{equation}
		is satisfied in $\mathbf T$.
		
		On the other hand, since $0 \leq a^*$ then $a \to 0 \leq a \to a^*$ by (\ref{180326_01}). This implies that $\sim (a \to a^*) \leq \sim a^*$ by (\ref{190326_02}). Then, by identity (\ref{290426_163}),
		$%\begin{equation} \label{290426_168}
		.\sim x^* \approx \sim (x \to x^*).
		$%\end{equation}
		Furthermore, by (R\ref{dobleneg}), 
		\begin{equation} \label{290426_169}
			x^* \approx x \to x^*
		\end{equation}
		verifies.
		
		The identity 
		$%\begin{equation} \label{290426_170}
		(x \wedge y) \to x^* \approx (x \wedge y)^*
		$%\end{equation}
		is satisfied since
		%$$
		%\begin{array}{lcll}
		%(a \wedge b) \to a^*	& = & (a \wedge b) \to (a \to 0) & \mbox{} \\
		%	& = & (a \wedge b) \to ((a \vee (a \wedge b)) \to 0) & \mbox{} \\
		%	& = & (a \wedge b) \to [(a \to 0) \wedge ((a \wedge b) \to 0)] & \mbox{by (R\ref{supremo})} \\
		%	& = & [(a \wedge b) \to (a \to 0)] \wedge [(a \wedge b) \to ((a \wedge b) \to 0)] & \mbox{by (R\ref{infimo})} \\
		%	& = & [(a \wedge b) \to (a \to 0)] \wedge [(a \wedge b) \to 0] & \mbox{by (\ref{290426_169})} \\
		%	& = & (a \wedge b) \to [(a \to 0) \wedge 0] & \mbox{by (R\ref{infimo})} \\
		%	& = & (a \wedge b) \to 0 & \mbox{} \\
		%	& = & (a \wedge b)^*. & \mbox{} 
		%\end{array}
		%$$
		\noindent $(a \wedge b) \to a^*	$
		$\overset{  
		}{=}  (a \wedge b) \to (a \to 0) $
		$\overset{  
		}{=}  (a \wedge b) \to ((a \vee (a \wedge b)) \to 0) $
		$\overset{  (R\ref{supremo}) 
		}{=}  (a \wedge b) \to [(a \to 0) \wedge ((a \wedge b) \to 0)] $
		$\overset{  (R\ref{infimo}) 
		}{=}  [(a \wedge b) \to (a \to 0)] \wedge [(a \wedge b) \to ((a \wedge b) \to 0)] $
		$\overset{  (\ref{290426_169}) 
		}{=}  [(a \wedge b) \to (a \to 0)] \wedge [(a \wedge b) \to 0] $
		$\overset{  (R\ref{infimo}) 
		}{=}  (a \wedge b) \to [(a \to 0) \wedge 0] $
		$\overset{  
		}{=}  (a \wedge b) \to 0 $
		$\overset{  
		}{=}  (a \wedge b)^*$
		and, as consequence, 
		\begin{equation} \label{290426_173}
			\mathbf T \models	(x \wedge y) \to y^* \approx (x \wedge y)^*.
		\end{equation}
		
		As
		$$
		\begin{array}{lcll}
			a \wedge \sim 0	& = & a & \mbox{since } \sim 0 = 1 \\
			& \leq & \sim a \vee a & \mbox{} \\
			& = & \sim a \vee \sim \sim a & \mbox{} \\
			& = & \sim (a \wedge \sim a) & \mbox{by (R\ref{negsup})} 
		\end{array}
		$$
		then by item (\ref{300326_01}), $(a \wedge \sim 0) \to \sim (a \wedge \sim a) = 1$. Also by  (R\ref{au}), $\sim a^* \to (a \wedge \sim 0) = 1$. Moreover $\sim a^* \to \sim (a \wedge \sim a) = 1$ by item (\ref{250426_12}). Then
		\begin{equation} \label{300426_01}
			\mathbf T \models	\sim x^* \to \sim (x \wedge \sim x) \approx 1.
		\end{equation}

		Besides,
		%$$
		%\begin{array}{lcll}
		%\sim (a \wedge a^*) \to \sim (a \wedge \sim a)	& = & (\sim a \vee \sim a^*) \to \sim (a \wedge \sim a) & \mbox{by (R\ref{negsup})} \\
		%	& = & [\sim a \to \sim (a \wedge \sim a)] \wedge [\sim a^* \to \sim (a \wedge \sim a)] & \mbox{by (R\ref{supremo})} \\
		%	& = & 1 \wedge  [\sim a^* \to \sim (a \wedge \sim a)] & \mbox{by (\ref{300326_01}) and  (\ref{190326_02})} \\
		%	& = & \sim a^* \to \sim (a \wedge \sim a) & \mbox{} \\
		%   & = & 1 & \mbox{by (\ref{300426_01})} 
		%\end{array}
		%$$
		\noindent $\sim (a \wedge a^*) \to \sim (a \wedge \sim a)	$
		$\overset{  (R\ref{negsup}) 
		}{=}  (\sim a \vee \sim a^*) \to \sim (a \wedge \sim a) $
		$\overset{  (R\ref{supremo}) 
		}{=}  [\sim a \to \sim (a \wedge \sim a)] \wedge [\sim a^* \to \sim (a \wedge \sim a)] $
		$\overset{  (\ref{300326_01}) and  (\ref{190326_02}) 
		}{=}  1 \wedge  [\sim a^* \to \sim (a \wedge \sim a)] $
		$\overset{  
		}{=}  \sim a^* \to \sim (a \wedge \sim a) $
		$\overset{  (\ref{300426_01}) 
		}{=}  1. $ 
		Then $\sim (a \wedge a^*) \to \sim (a \wedge \sim a) = 1$. Also, by (\ref{270426_16}),  $(a \wedge \sim a) \to (a \wedge a^*) = 1$. Applying item (\ref{210426_01}), we have that $$a \wedge \sim a \leq a \wedge a^*.$$
		By (\ref{250426_10}), $a \wedge a^* \leq \sim a$, so $$a \wedge a^* \leq a \wedge \sim a.$$
		Then
		\begin{equation} \label{290426_140}
			\mathbf T \models	x \wedge \sim x \approx x \wedge x^*.
		\end{equation}
		
		Besides,
		\begin{equation} \label{290426_141}
			(x \to \sim y) \wedge (x \to (\sim y)^*) \approx (x \to  y) \wedge (x \to  y^*)	.
		\end{equation}
		is satisfied in $\mathbf T$ because
		%$$
		%\begin{array}{lcll}
		%(a \to \sim b) \wedge (a \to (\sim b)^*)	& = & a \to (\sim b \wedge (\sim b)^*) & \mbox{by (R\ref{infimo})} \\
		%	& = & a \to (\sim b \wedge (\sim \sim b)) & \mbox{by (\ref{290426_140})} \\
		%	& = & a \to (\sim b \wedge b) & \mbox{by (R\ref{dobleneg})} \\
		%	& = & a \to (b \wedge b^*) & \mbox{by (\ref{290426_140})} \\
		%	& = & (a \to  b) \wedge (a \to  b^*) & \mbox{by (R\ref{infimo})} 
		%\end{array}
		%$$ 
		\noindent $(a \to \sim b) \wedge (a \to (\sim b)^*)	$
		$\overset{  (R\ref{infimo}) 
		}{=}  a \to (\sim b \wedge (\sim b)^*) $
		$\overset{  (\ref{290426_140}) 
		}{=}  a \to (\sim b \wedge (\sim \sim b)) $
		$\overset{  (R\ref{dobleneg}) 
		}{=}  a \to (\sim b \wedge b) $
		$\overset{  (\ref{290426_140}) 
		}{=}  a \to (b \wedge b^*) $
		$\overset{  (R\ref{infimo}) 
		}{=}  (a \to  b) \wedge (a \to  b^*) $.
		
		Since
		%$$
		%\begin{array}{lcll}
		%(a \wedge b) \to (\sim b)	& = & [(a \wedge b) \to (\sim b)] \wedge 1 & \mbox{} \\
		%	& = & [(a \wedge b) \to (\sim b)] \wedge [(a \wedge b) \to b] & \mbox{by (\ref{300326_01})} \\
		%	& = & (a \wedge b) \to ((\sim b) \wedge b) & \mbox{by (R\ref{infimo})} \\
		%	& = & (a \wedge b) \to ((\sim b) \wedge (\sim \sim b)) & \mbox{by (R\ref{dobleneg})} \\
		%	& = & (a \wedge b) \to ((\sim b) \wedge (\sim b)^*) & \mbox{by (\ref{290426_140})} \\
		%	& = & [(a \wedge b) \to (\sim b)] \wedge [(a \wedge b) \to (\sim b)^*] & \mbox{by (R\ref{infimo})} \\
		%	& = & [(a \wedge b) \to b] \wedge [(a \wedge b) \to b^*] & \mbox{by (\ref{290426_141})} \\
		%	& = & 1 \wedge [(a \wedge b) \to b^*] & \mbox{by (\ref{300326_01})} \\
		%	& = & (a \wedge b) \to b^* & \mbox{} \\
		%	& = & (a \wedge b)^* & \mbox{by (\ref{290426_173})} 
		%\end{array}
		%$$
		\noindent $(a \wedge b) \to (\sim b)	$
		$\overset{  
		}{=}  [(a \wedge b) \to (\sim b)] \wedge 1 $
		$\overset{  (\ref{300326_01}) 
		}{=}  [(a \wedge b) \to (\sim b)] \wedge [(a \wedge b) \to b] $
		$\overset{  (R\ref{infimo}) 
		}{=}  (a \wedge b) \to ((\sim b) \wedge b) $
		$\overset{  (R\ref{dobleneg}) 
		}{=}  (a \wedge b) \to ((\sim b) \wedge (\sim \sim b)) $
		$\overset{  (\ref{290426_140}) 
		}{=}  (a \wedge b) \to ((\sim b) \wedge (\sim b)^*) $
		$\overset{  (R\ref{infimo}) 
		}{=}  [(a \wedge b) \to (\sim b)] \wedge [(a \wedge b) \to (\sim b)^*] $
		$\overset{  (\ref{290426_141}) 
		}{=}  [(a \wedge b) \to b] \wedge [(a \wedge b) \to b^*] $
		$\overset{  (\ref{300326_01}) 
		}{=}  1 \wedge [(a \wedge b) \to b^*] $
		$\overset{  
		}{=}  (a \wedge b) \to b^* $
		$\overset{  (\ref{290426_173}) 
		}{=}  (a \wedge b)^* $ 
		then
		\begin{equation} \label{290426_182}
			(x \wedge y) \to (\sim y) \approx (x \wedge y)^*.
		\end{equation}
		is satisfied in $\mathbf T$.
		
		Since $a \wedge b \leq a$ then $a \to c \leq (a \wedge b) \to c$ by (\ref{190326_04}). Thus, using (\ref{190326_02}), $\sim ((a \wedge b) \to c) \leq \sim (a \to c)$. Consequently $\sim ((a \wedge b) \to c)  \to \sim (a \to c) = 1$ by item (\ref{300326_01}). Now we will prove that
		\begin{equation} \label{290426_154}
			\mathbf T \models \sim ((x \wedge y) \to z)  \to \sim (x \to z) \approx 1.
		\end{equation}
		This item is proved as follows:
		
		%$$
		%\begin{array}{lcll}
		%(a \wedge \sim b) \to \sim (a \to b)	& = &  \sim (a \wedge \sim b)^* \to \sim (a \to b)& \mbox{by (\ref{290426_157})} \\
		%	& = & \sim [(a \wedge \sim b) \to (\sim \sim b)] \to \sim (a \to b) & \mbox{by (\ref{290426_182})} \\
		%	& = & \sim [(a \wedge \sim b) \to b] \to \sim (a \to b) & \mbox{by (R\ref{dobleneg})} \\
		%	& = & 1 & \mbox{by (\ref{290426_154})} 
		%\end{array}
		%$$
		\noindent $(a \wedge \sim b) \to \sim (a \to b)	$
		$\overset{  (\ref{290426_157}) 
		}{=}   \sim (a \wedge \sim b)^* \to \sim (a \to b)$
		$\overset{  (\ref{290426_182}) 
		}{=}  \sim [(a \wedge \sim b) \to (\sim \sim b)] \to \sim (a \to b) $
		$\overset{  (R\ref{dobleneg}) 
		}{=}  \sim [(a \wedge \sim b) \to b] \to \sim (a \to b) $
		$\overset{  (\ref{290426_154}) 
		}{=}  1 $.

	\end{itemize}	
\end{proof}

The following result is an immediate consequence of Lemmas \ref{lemaDistrib} and \ref{otherconditions}.

\begin{theorem}
$\RNA = \mathcal R$.
\end{theorem}

\section{Independence of the identities of $\mathcal R$}

In this section, we show that the identities given in Definition \ref{def_R} are independent. 
In order to establish this fact, we present several examples obtained using a computational tool \cite{prover9}.
\vspace{1pt}

We begin with the following examples.

\begin{itemize}

	\item \noindent $\mathbf A_{\ref{lattice2}} = \langle \{0,1\}; \wedge, \vee, \to, \sim, 0, 1 \rangle$ where 

%\noindent \begin{minipage}{0.5 \textwidth}
	\begin{tabular}{r|r|r|}
		$\sim$ & 0 & 1 \\ \hline
		& 0 & 1
	\end{tabular} \hspace{.5cm}
	\begin{tabular}{r|r|r|}
		$\to$ & 0 & 1 \\ \hline
		0 & 1 & 1 \\ \hline
		1 & 1 & 1
	\end{tabular} \hspace{.5cm}
%\end{minipage}
%\begin{minipage}{0.8 \textwidth}
	\begin{tabular}{r|r|r|}
		$\wedge$ & 0 & 1 \\ \hline
		0 & 0 & 0 \\ \hline
		1 & 1 & 1
	\end{tabular} \hspace{.5cm}
	\begin{tabular}{r|r|r|}
		$\vee$ & 0 & 1 \\ \hline
		0 & 0 & 0 \\ \hline
		1 & 1 & 1
	\end{tabular}
%\end{minipage}

\item \noindent $\mathbf A_{\ref{top}} = \langle \{0,b,1\}; \wedge, \vee, \to, \sim, 0, 1 \rangle$ where

%\noindent \begin{minipage}{0.5 \textwidth}
	\begin{tabular}{r|r|r|r|}
		$\sim$ & 0 & 1 & b \\ \hline
		& b & 1 & 0
	\end{tabular} \hspace{.5cm}
	\begin{tabular}{r|r|r|r|}
		$\to$ & 0 & 1 & b \\ \hline
		0 & 1 & 1 & 1 \\ \hline
		1 & 1 & 1 & 1 \\ \hline
		b & 0 & 1 & 1
	\end{tabular} \hspace{.5cm}
%\end{minipage}
%\begin{minipage}{0.8 \textwidth}
	\begin{tabular}{r|r|r|r|}
		$\wedge$ & 0 & 1 & b \\ \hline
		0 & 0 & 0 & 0 \\ \hline
		1 & 0 & 1 & 1 \\ \hline
		b & 0 & 1 & b
	\end{tabular} \hspace{.5cm}
	\begin{tabular}{r|r|r|r|}
		$\vee$ & 0 & 1 & b \\ \hline
		0 & 0 & 1 & b \\ \hline
		1 & 1 & 1 & b \\ \hline
		b & b & b & b
	\end{tabular}
%\end{minipage}

\item \noindent $\mathbf A_{\ref{bottom}} = \langle \{0,b,1\}; \wedge, \vee, \to, \sim, 0, 1 \rangle$ where 

%\noindent \begin{minipage}{0.5 \textwidth}
	\begin{tabular}{r|r|r|r|}
		$\sim$ & 0 & 1 & b \\ \hline
		& 0 & b & 1
	\end{tabular} \hspace{.5cm}
	\begin{tabular}{r|r|r|r|}
		$\to$ & 0 & 1 & b \\ \hline
		0 & 1 & 1 & 1 \\ \hline
		1 & 0 & 1 & b \\ \hline
		b & 1 & 1 & 1
	\end{tabular} \hspace{.5cm}
%\end{minipage}
%\begin{minipage}{0.8 \textwidth}
	\begin{tabular}{r|r|r|r|}
		$\wedge$ & 0 & 1 & b \\ \hline
		0 & 0 & 0 & b \\ \hline
		1 & 0 & 1 & b \\ \hline
		b & b & b & b
	\end{tabular} \hspace{.5cm}
	\begin{tabular}{r|r|r|r|}
		$\vee$ & 0 & 1 & b \\ \hline
		0 & 0 & 1 & 0 \\ \hline
		1 & 1 & 1 & 1 \\ \hline
		b & 0 & 1 & b
	\end{tabular}
%\end{minipage}

\item \noindent $\mathbf A_{\ref{dobleneg}} = \langle \{0,1\}; \wedge, \vee, \to, \sim, 0, 1 \rangle$ where

%\noindent \begin{minipage}{0.5 \textwidth}
	\begin{tabular}{r|r|r|}
		$\sim$ & 0 & 1 \\ \hline
		& 1 & 1
	\end{tabular} \hspace{.5cm}
	\begin{tabular}{r|r|r|}
		$\to$ & 0 & 1 \\ \hline
		0 & 1 & 1 \\ \hline
		1 & 1 & 1
	\end{tabular} \hspace{.5cm}
%\end{minipage}
%\begin{minipage}{0.8 \textwidth}
	\begin{tabular}{r|r|r|}
		$\wedge$ & 0 & 1 \\ \hline
		0 & 0 & 0 \\ \hline
		1 & 0 & 1
	\end{tabular} \hspace{.5cm}
	\begin{tabular}{r|r|r|}
		$\vee$ & 0 & 1 \\ \hline
		0 & 0 & 1 \\ \hline
		1 & 1 & 1
	\end{tabular}
%\end{minipage}

\item \noindent $\mathbf A_{\ref{negsup}} = \langle \{0,1\}; \wedge, \vee, \to, \sim, 0, 1 \rangle$ where

%\noindent \begin{minipage}{0.5 \textwidth}
	\begin{tabular}{r|r|r|}
		$\sim$ & 0 & 1 \\ \hline
		& 0 & 1
	\end{tabular} \hspace{.5cm}
	\begin{tabular}{r|r|r|}
		$\to$ & 0 & 1 \\ \hline
		0 & 1 & 1 \\ \hline
		1 & 1 & 1
	\end{tabular} \hspace{.5cm}
%\end{minipage}
%\begin{minipage}{0.8 \textwidth}
	\begin{tabular}{r|r|r|}
		$\wedge$ & 0 & 1 \\ \hline
		0 & 0 & 0 \\ \hline
		1 & 0 & 1
	\end{tabular} \hspace{.5cm}
	\begin{tabular}{r|r|r|}
		$\vee$ & 0 & 1 \\ \hline
		0 & 0 & 1 \\ \hline
		1 & 1 & 1
	\end{tabular}
%\end{minipage}

\item \noindent $\mathbf A_{\ref{supremo}} = \langle \{0,b,c,1\}; \wedge, \vee, \to, \sim, 0, 1 \rangle$ where

%\noindent \begin{minipage}{0.5 \textwidth}
	\begin{tabular}{r|r|r|r|r|}
		$\sim$ & 0 & 1 & b & c \\ \hline
		& 1 & 0 & b & c
	\end{tabular} \hspace{.5cm}
	\begin{tabular}{r|r|r|r|r|}
		$\to$ & 0 & 1 & b & c \\ \hline
		0 & 1 & 1 & 1 & 1 \\ \hline
		1 & 0 & 1 & b & c \\ \hline
		b & 1 & 1 & 1 & 1 \\ \hline
		c & 1 & 1 & 1 & 1
	\end{tabular} \hspace{.5cm}
%\end{minipage}
%\begin{minipage}{0.8 \textwidth}
	\begin{tabular}{r|r|r|r|r|}
		$\wedge$ & 0 & 1 & b & c \\ \hline
		0 & 0 & 0 & 0 & 0 \\ \hline
		1 & 0 & 1 & b & c \\ \hline
		b & 0 & b & b & 0 \\ \hline
		c & 0 & c & 0 & c
	\end{tabular} \hspace{.5cm}
	\begin{tabular}{r|r|r|r|r|}
		$\vee$ & 0 & 1 & b & c \\ \hline
		0 & 0 & 1 & b & c \\ \hline
		1 & 1 & 1 & 1 & 1 \\ \hline
		b & b & 1 & b & 1 \\ \hline
		c & c & 1 & 1 & c
	\end{tabular}
%\end{minipage}

\item \noindent $\mathbf A_{\ref{infimo}} = \langle \{0,b,c,d,e,1\}; \wedge, \vee, \to, \sim, 0, 1 \rangle$ where

%\noindent \begin{minipage}{0.5 \textwidth}
	\begin{tabular}{r|r|r|r|r|r|r|}
		$\sim$ & 0 & 1 & b & c & d & e \\ \hline
		& 1 & 0 & c & b & e & d
	\end{tabular} \hspace{.5cm}
	\begin{tabular}{r|r|r|r|r|r|r|}
		$\to$ & 0 & 1 & b & c & d & e \\ \hline
		0 & 1 & 1 & 1 & 1 & 1 & 1 \\ \hline
		1 & 0 & 1 & b & c & d & e \\ \hline
		b & d & 1 & 1 & c & d & 1 \\ \hline
		c & b & 1 & b & 1 & e & e \\ \hline
		d & 1 & 1 & 1 & 1 & 1 & 1 \\ \hline
		e & d & 1 & 1 & c & d & 1
	\end{tabular} \hspace{.5cm}

%\end{minipage}
%\begin{minipage}{0.8 \textwidth}
	\begin{tabular}{r|r|r|r|r|r|r|}
		$\wedge$ & 0 & 1 & b & c & d & e \\ \hline
		0 & 0 & 0 & 0 & 0 & 0 & 0 \\ \hline
		1 & 0 & 1 & b & c & d & e \\ \hline
		b & 0 & b & b & 0 & 0 & b \\ \hline
		c & 0 & c & 0 & c & d & d \\ \hline
		d & 0 & d & 0 & d & d & d \\ \hline
		e & 0 & e & b & d & d & e
	\end{tabular} \hspace{.5cm}
	\begin{tabular}{r|r|r|r|r|r|r|}
		$\vee$ & 0 & 1 & b & c & d & e \\ \hline
		0 & 0 & 1 & b & c & d & e \\ \hline
		1 & 1 & 1 & 1 & 1 & 1 & 1 \\ \hline
		b & b & 1 & b & 1 & e & e \\ \hline
		c & c & 1 & 1 & c & c & 1 \\ \hline
		d & d & 1 & e & c & d & e \\ \hline
		e & e & 1 & e & 1 & e & e
	\end{tabular}
%\end{minipage}

\item \noindent $\mathbf A_{\ref{trans}} = \langle \{0,b,1\}; \wedge, \vee, \to, \sim, 0, 1 \rangle$ where

%\noindent \begin{minipage}{0.5 \textwidth}
	\begin{tabular}{r|r|r|r|}
		$\sim$ & 0 & 1 & b \\ \hline
		& 1 & 0 & b
	\end{tabular} \hspace{.5cm}
	\begin{tabular}{r|r|r|r|}
		$\to$ & 0 & 1 & b \\ \hline
		0 & 1 & 1 & 1 \\ \hline
		1 & 0 & 1 & b \\ \hline
		b & b & 1 & 1
	\end{tabular} \hspace{.5cm}
%\end{minipage}
%\begin{minipage}{0.5 \textwidth}
	\begin{tabular}{r|r|r|r|}
		$\wedge$ & 0 & 1 & b \\ \hline
		0 & 0 & 0 & 0 \\ \hline
		1 & 0 & 1 & b \\ \hline
		b & 0 & b & b
	\end{tabular} \hspace{.5cm}
	\begin{tabular}{r|r|r|r|}
		$\vee$ & 0 & 1 & b \\ \hline
		0 & 0 & 1 & b \\ \hline
		1 & 1 & 1 & 1 \\ \hline
		b & b & 1 & b
	\end{tabular}
%\end{minipage}

\item \noindent $\mathbf A_{\ref{inf}} = \langle \{0,1\}; \wedge, \vee, \to, \sim, 0, 1 \rangle$ where 

%\noindent \begin{minipage}{0.5 \textwidth}
	\begin{tabular}{r|r|r|}
		$\sim$ & 0 & 1 \\ \hline
		& 0 & 1
	\end{tabular} \hspace{.5cm}
	\begin{tabular}{r|r|r|}
		$\to$ & 0 & 1 \\ \hline
		0 & 1 & 1 \\ \hline
		1 & 1 & 1
	\end{tabular} \hspace{.5cm}
%\end{minipage}
%\begin{minipage}{0.5 \textwidth}
	\begin{tabular}{r|r|r|}
		$\wedge$ & 0 & 1 \\ \hline
		0 & 0 & 0 \\ \hline
		1 & 0 & 1
	\end{tabular} \hspace{.5cm}
	\begin{tabular}{r|r|r|}
		$\vee$ & 0 & 1 \\ \hline
		0 & 0 & 0 \\ \hline
		1 & 0 & 1
	\end{tabular}
%\end{minipage}

\item \noindent $\mathbf A_{\ref{au}} = \langle \{0,b,1\}; \wedge, \vee, \to, \sim, 0, 1 \rangle$ where 

%\noindent \begin{minipage}{0.5 \textwidth}
	\begin{tabular}{r|r|r|r|}
		$\sim$ & 0 & 1 & b \\ \hline
		& 1 & 0 & b
	\end{tabular} \hspace{.5cm}
	\begin{tabular}{r|r|r|r|}
		$\to$ & 0 & 1 & b \\ \hline
		0 & 1 & 1 & 1 \\ \hline
		1 & 0 & 1 & 0 \\ \hline
		b & 0 & 1 & 0
	\end{tabular} \hspace{.5cm}
%\end{minipage}
%\begin{minipage}{0.8 \textwidth}
	\begin{tabular}{r|r|r|r|}
		$\wedge$ & 0 & 1 & b \\ \hline
		0 & 0 & 0 & 0 \\ \hline
		1 & 0 & 1 & b \\ \hline
		b & 0 & b & b
	\end{tabular} \hspace{.5cm}
	\begin{tabular}{r|r|r|r|}
		$\vee$ & 0 & 1 & b \\ \hline
		0 & 0 & 1 & b \\ \hline
		1 & 1 & 1 & 1 \\ \hline
		b & b & 1 & b
	\end{tabular}
%\end{minipage}

\end{itemize}

It can be proved that each algebra $\mathbf A_j$ satisfies the axiom $\mathrm{(Rk)}$ of Definition \ref{def_R} for 
$k \not= j$, but does not satisfy $\mathrm{(Rj)}$. Therefore, we obtain the following result.

\begin{lemma}
The identities appearing in Definition \ref{def_R} are independent.
\end{lemma}

\section*{Acknowledgements} We thank Ignacio Viglizzo for some discussions in the early stages of this research.

%%%%%%%%%%%%%%%%%%%%%%%%%%%%%%%%%%%%%%%%%%%%%%%
%\newpage
-----------------------------------------------------------------------------------------
%\newpage
%\begin{multicols}{3}
	\begin{enumerate}
		\item[] Juan Manuel Cornejo,\\
		Departamento de Matem\'atica \\
		(Universidad Nacional del Sur) and \\
		INMABB (CONICET).\\
		Bah\'ia Blanca (8000), Argentina.\\
		jmcornejo@.uns.edu.ar
		
		\item[] Paula Soledad Helt,\\
		Departamento de Matem\'atica \\
		(Universidad Nacional del Sur) and \\
		INMABB (CONICET).\\
		Bah\'ia Blanca (8000), Argentina.\\
		paulasoledadhelt@gmail.com
		
		\item[] Hern\'an Javier San Mart\'in,\\
		Centro de Matem\'atica de La Plata (CMaLP), \\
		Facultad de Ciencias Exactas (UNLP), \\
		and CONICET.\\
		La Plata (1900), Argentina.\\
		hsanmartin@mate.unlp.edu.ar
	\end{enumerate}
%\end{multicols}

\end{document}